\documentclass[12pt,svgnames]{article}
\usepackage {graphics} 
\usepackage {graphicx}  
\usepackage {pstricks,pst-node,  pst-coil} 
\onecolumn  
\usepackage{amsmath,amssymb,latexsym}
\usepackage{amsmath,amsthm}
\usepackage{xcolor}
\usepackage{subcaption}
\usepackage{tikz}

\pgfdeclarelayer{edgelayer}
\pgfdeclarelayer{nodelayer}
\pgfsetlayers{edgelayer,nodelayer,main}

\tikzstyle{black_dot}=[fill=black, draw=black, shape=circle, inner sep=0pt, minimum size=3pt]
\tikzstyle{white_circle}=[fill=white, draw=black, shape=circle, inner sep=0pt, minimum size=3pt]
\tikzstyle{empty}=[fill=none, draw=none, shape=rectangle, minimum size=3pt, inner sep=0pt]
\tikzstyle{normal_edge}=[very thick]
\tikzstyle{dotted_edge}=[-, dotted, very thick]
\tikzstyle{blue_edge}=[-, draw=blue, dashed, very thick]
\tikzstyle{red_edge}=[-, draw=red, very thick]
\tikzstyle{green_edge}=[-, draw=Green, dotted, very thick]
\tikzstyle{yellow_edge}=[-, draw=yellow, dash dot, very thick]
\tikzstyle{right_arrow}=[fill=none, ->, very thick]
\tikzstyle{dashed_edge}=[-, dashed, very thick]
\tikzstyle{red_dashed}=[-, dotted, draw=red, very thick]  
\tikzstyle{blue_dashed}=[-, dotted, draw=blue, very thick] 
  
\newtheorem{tw}{Theorem}  
\newtheorem{lem}[tw]{Lemma}

\newtheorem{cnj}[tw]{Conjecture}  
\newtheorem{cor}[tw]{Corollary}

\begin{document}
\hyphenation{every}
\title{On Local Irregularity Conjecture for 2-multigraphs}
\author{Igor Grzelec\thanks{Department of Discrete Mathematics, AGH University of Krakow, Poland} \thanks{The corresponding author. Email:  grzelec@agh.edu.pl}, Alfréd Onderko\thanks{Institute of mathematics, P.J. Šafárik University, Košice, Slovakia}, Mariusz Woźniak\footnotemark[1]}

\maketitle

\begin{abstract}
A multigraph in which adjacent vertices have different degrees is called \textit{locally irregular}.
The \textit{locally irregular edge coloring} is an edge coloring of a multigraph $G$ in which every color induces a locally irregular submultigraph of $G$. 
We denote by $\operatorname{lir}(G)$ the \textit{locally irregular chromatic index} of a multigraph $G$, which is the smallest number of colors required in a locally irregular edge coloring of $G$, given that such a coloring of $G$ exists.
By $^2G$ we denote a 2-multigraph obtained from a simple graph $G$ by doubling each its edge.
In 2022 Grzelec and Woźniak conjectured that $\operatorname{lir}(^2G) \leq 2$ for every connected simple graph $G$ different from $K_2$; the conjecture is known as Local Irregularity Conjecture for 2-multigraphs. 
In this paper, we prove this conjecture in the case of regular graphs, split graphs, and some particular families of subcubic graphs. 
Moreover, we provide a constant upper bound on the locally irregular chromatic index of planar 2-multigraphs (except for $^2K_2$), and we obtain a better constant upper bound on $\operatorname{lir}(^2G)$ if $G$ is a simple subcubic graph different from $K_2$.
In the proofs, special decompositions of graphs and the relation of Local Irregularity Conjecture to the well-known 1-2-3 Conjecture are utilized.\\

\end{abstract}

\section{Introduction}
All graphs in this paper are finite and without loops.
In this paper, the primary focus is on the multigraphs with a multiplicity of each edge being precisely two.
However, due to the studied problem itself and the need for the use of submultigraphs of these specific multigraphs, we introduce most of the definitions in the most general way - for general multigraphs.
Formally, a multigraph $G$ consists of the set of vertices $V(G)$ - a nonempty finite set, and of the multiset of edges $E(G)$ - the multiset of 2-element subsets of $V(G)$.
The multiplicity of each edge $e$ in $G$, the number of occurrences of $e$ in $E(G)$, is denoted by $\mu_G(e)$.
Hence, if $\mu_G(e) = 0$ then $e$ is not an edge of $G$.
If $\mu_G(e) \geq 2$, all occurrences of $e$ in the multiset $E(G)$ are said to be \textit{parallel edges forming the multiedge} $e$.
If the multiplicity of each edge of $G$ is at most one then we say that $G$ is a simple graph, and if the multiplicity of each edge is zero or two then we say that $G$ is a 2-multigraph.
Hence a 2-multigraph may be constructed from a simple graph by replacing each edge with two parallel edges; 
we use $^2G$ to simply denote the 2-multigraph obtained from a simple graph $G$ by replacing each edge of $G$ with two parallel edges.
The degree of a vertex $v$ is the sum of multiplicities of edges incident to $v$.

A graph is \textit{locally irregular} if the neighboring vertices have different degrees. 
An edge coloring of a multigraph $G$ is an assignment of colors to elements of the multiset $E(G)$.
If each color of an edge coloring of $G$ induces a locally irregular submultigraph, the edge coloring is called \textit{locally irregular}.
Moreover, if no more than $k$ colors are used in a locally irregular coloring of a graph, we say that the coloring is locally irregular $k$-edge coloring.
The smallest number $k$ such that a graph $G$ has a locally irregular $k$-coloring is the locally irregular chromatic index of $G$, denoted by $\operatorname{lir}(G)$.

In most cases, we use colors red, blue, green, etc., so it is easy to illustrate the considered colorings in figures.
Moreover, in the case of multiedges with multiplicity two, the parallel edges forming a multiedge need not be colored with the same color. 
Hence we say that the multiedge $e$ is colored for example red-blue, if one of the parallel edges forming $e$ is colored red and the second one is blue. 
Similarly, if we say that the multiedge $e$ of a 2-multigraph is colored red-red then both parallel edges forming $e$ are red.

Locally irregular colorings were mostly studied in the case of simple graphs, however in~\cite{Grzelec Madaras Onderko Roman Sotak}, \cite{Grzelec wozniak}, and \cite{Grzelec wozniak2} locally irregular colorings of multigraphs were considered. 
Naturally, if $G$ is a simple graph with $\operatorname{lir}(G) \leq k$ then $\operatorname{lir}(^2G) \leq k$, as we may replace each edge of color $c$ in the coloring of $G$ with two parallel edges of color $c$ in $^2G$, and obtain a locally irregular coloring of~$^2G$. 

Note that not every simple graph is locally irregular. 
As we almost exclusively deal with locally irregular colorings, we say $G$ is \textit{colorable} if it admits a locally irregular coloring with any number of colors.
If $G$ is not colorable, we say it is \textit{uncolorable}.
In \cite{Baudon Bensmail Przybylo Wozniak} Baudon, Bensmail, Przybyło, and Woźniak fully characterized the family $\mathfrak{T'}$ of all simple uncolorable graphs.
Interestingly, all simple connected uncolorable graphs are special subcubic cacti, and, except for odd-length cycles, each such graph contains only cycles of length three.
Authors of~\cite{Baudon Bensmail Przybylo Wozniak} also conjectured that every simple connected graph $G \notin \mathfrak{T'}$ satisfies $\operatorname{lir}(G)\leq 3$.
After providing a counterexample to this conjecture i.e. the bow-tie graph which needs four colors in a locally irregular edge coloring~\cite{Sedlar Skrekovski}, Sedlar and \v Skrekovski restated the conjecture:
\begin{cnj}[Local Irregularity Conjecture~\cite{Baudon Bensmail Przybylo Wozniak},~\cite{Sedlar Skrekovski 2}]
\label{graph3}
    Every simple connected locally irregular colorable graph $G$ different from the bow-tie graph satisfies $\operatorname{lir}(G)\leq 3$.
\end{cnj}

Since all simple connected uncolorable graphs are special subcubic cacti, a lot of attention was given to study the locally irregular indices of graphs from these superclasses.  
In~\cite{Sedlar Skrekovski 2} it was proved that $\operatorname{lir}(G) \leq 3$ for every colorable cactus $G$ different from the bow-tie graph.
For claw-free subcubic colorable graphs, the Conjecture~\ref{graph3} was confirmed in~\cite{Luzar Macekova} by Lužer et al.

\begin{figure}[h!]
\centering
    \begin{tikzpicture}
    \begin{pgfonlayer}{nodelayer}
    \node [style={black_dot}] (0) at (-0.5, 0) {};
    \node [style={black_dot}] (1) at (0.5, 0) {};
    \node [style={black_dot}] (2) at (-0.5, 1) {};
    \node [style={black_dot}] (3) at (-1.5, 0.5) {};
    \node [style={black_dot}] (4) at (-1.5, -0.5) {};
    \node [style={black_dot}] (5) at (-0.5, -1) {};
    \node [style={black_dot}] (6) at (0.5, -1) {};
    \node [style={black_dot}] (7) at (0.5, 1) {};
    \node [style={black_dot}] (8) at (1.5, 0.5) {};
    \node [style={black_dot}] (9) at (1.5, -0.5) {};
    \end{pgfonlayer}
    \begin{pgfonlayer}{edgelayer}
    \draw [style={red_edge}] (2) to (0);
    \draw [style={red_edge}] (0) to (1);
    \draw [style={red_edge}] (0) to (5);
    \draw [style={blue_edge}] (5) to (4);
    \draw [style={blue_edge}] (4) to (0);
    \draw [style={blue_edge}] (1) to (9);
    \draw [style={blue_edge}] (9) to (6);
    \draw [style={green_edge}] (0) to (3);
    \draw [style={green_edge}] (3) to (2);
    \draw [style={green_edge}] (7) to (8);
    \draw [style={green_edge}] (8) to (1);
    \draw [style={yellow_edge}] (7) to (1);
    \draw [style={yellow_edge}] (1) to (6);
    \end{pgfonlayer}
    \end{tikzpicture}
\caption{The bow-tie graph $B$ and its locally irregular coloring with four colors.}
\label{bow-tie graph}
\end{figure}
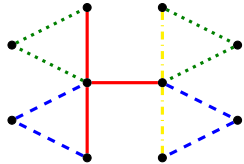

Conjecture~\ref{graph3} was also confirmed for simple graphs from other classes, namely trees in~\cite{Baudon Bensmail}, split graphs in~\cite{Lintzmayer Mota Sambinelli}, complete $k$-partite graphs where $k>1$ and powers of cycles in~\cite{Grzelec Madaras Onderko Roman Sotak}, $r$-regular graphs where $r\geq 10^7$ in~\cite{Baudon Bensmail Przybylo Wozniak}, and graphs with minimum degree at least $10^{10}$ in~\cite{Przybylo}. 
For a simple planar graph $G$, it was proven in \cite{Bensmail Dross Nisse} that $\operatorname{lir}(G)\leq 15$. 
The constant upper bound 328 on locally irregular index of a general colorable simple graph was given by Bensmail, Merker and Thomassen in~\cite{Bensmail Merker Thomassen}, and later lowered to 220 by Lu\v zar, Przybyło and Soták in \cite{Luzar Przybylo Sotak}.

The locally irregular colorings of simple graphs were inspired by \textit{neighbor-sum-distinguishing edge colorings} and multiset neighbor distinguishing edge colorings. We introduce the first of these concepts, as it is useful for proving new results on the locally irregular colorings of 2-multigraphs. \textit{Neighbor-sum-distinguishing edge colorings} are colorings by integers in which sums of colors used on edges incident to neighboring vertices are different. The smallest number of colors for which a neighbor-sum-distinguishing edge coloring exists, the \textit{neighbor-sum-distinguishing index}, is denoted by $\chi_{\Sigma}(G)$. In~\cite{Karonski Luczak Thomason}, Karoński, Łuczak, and Thomason introduced the problem of determining $\chi_{\Sigma}(G)$ and proposed the well-known 1-2-3 Conjecture which states that colors 1, 2 and 3 are enough to distinguish by sums the neighbors in every simple graph without an isolated edge.
Recently, in 2023, the 1-2-3 Conjecture was confirmed by Keusch:
\begin{tw}[\cite{Keusch}]\label{123}
    For every simple graph $G$ without isolated edges there is a neighbor-sum-distinguishing edge coloring of $G$ with colors $1$, $2$ and $3$.
\end{tw}
We utilize this result of Keusch to derive new results on the locally irregular chromatic index of regular 2-multigraphs in Section~\ref{regular_section}; we prove the following conjecture, proposed by Grzelec and Woźniak, in the case of regular 2-multigraphs:
\begin{cnj}[Local Irregularity Conjecture for 2-multigraphs \cite{Grzelec wozniak}]
\label{main}
For every connected graph $G$ which is not isomorphic to $K_2$ we have $\operatorname{lir}(^2G)\leq 2$.
\end{cnj}
In~\cite{Grzelec wozniak} and~\cite{Grzelec wozniak2} Local Irregularity Conjecture was proved for graphs of several classes:
\begin{tw}[\cite{Grzelec wozniak}, \cite{Grzelec wozniak2}]
\label{cycle}
The Local Irregularity Conjecture for $2$-multigraphs holds for: paths, cycles, wheels $W_n$, complete graphs $K_n$, for $n \geq 3$, bipartite graphs, complete k-partite graphs, for $k \geq 3$ and cacti.
\end{tw}
Apart from that, it is proved in \cite{Grzelec wozniak} that every connected graph $G$ satisfies $\operatorname{lir}(^2G)\leq 76$ if $G$ is not isomorphic to $K_2$.

In addition to proving Conjecture~\ref{main} for regular graphs in Section~\ref{regular_section}, we prove it in the case of split graphs in Section~\ref{split_section}, and the case of special subcubic graphs in Section~\ref{subcubic_section}.

Locally irregular colorings of planar 2-multigraphs are studied in Section~\ref{planar_section}, and the upper bound of four is proved for the locally irregular chromatic index of a connected planar 2-multigraph different from $^2K_2$. In the case of planar 2-multigraphs, this result greatly improves the general upper bound of 76 proved by Grzelec and Wozńiak in~\cite{Grzelec wozniak} and the upper bound of 15 implied by the result of Bensmail, Dross, and Nisse~\cite{Bensmail Dross Nisse} for simple colorable planar graphs.

In Section~\ref{subcubic_section} we prove that $\operatorname{lir}(^2G) \leq 3$ if $G$ is a simple subcubic graph without isolated edges, which also improves the bound of four given by the result of Lužar, Przybyło, and Soták~\cite{Luzar Przybylo Sotak} for simple colorable subcubic graphs.

\section{Regular graphs}\label{regular_section}

We present a short proof that the Conjecture \ref{main} holds for regular 2-multigraphs. The proof is based on Theorem \ref{123}, recently proved by Keusch in \cite{Keusch}.

\begin{tw}
If $G$ is a connected regular simple graph not isomorphic to $K_2$ then $\operatorname{lir}(^2G)\leq 2$.
\end{tw}
\begin{proof}
Let $G$ be $d$-regular, where $d>1$. 
From Theorem~\ref{123} we obtain a neighbor-sum-distinguishing edge coloring $\varphi$ of $G$ using colors 1, 2, and 3. 
Clearly, an edge coloring $\psi$ of $G$ such that $\psi(e)=\varphi(e)-1$ is neighbor-sum-distinguishing, too.
From $\psi$, we create a locally irregular edge coloring of $^2G$ in the following way:
every edge of color 0 is replaced by a blue-blue multiedge, every edge of color 1 is replaced by a red-blue multiedge, and every edge of color 2 is replaced by a red-red multiedge. 
Note that the number of red edges incident to a vertex equals the sum of colors used on those edges in $\psi$. 
Hence, from the fact that $\psi$ is neighbor-sum-distinguishing, we get that the red degrees of every two adjacent vertices in the coloring of $^2G$ are different.
Since $^2G$ is $2d$-regular, two vertices have different blue degrees whenever they have different red degrees.
Thus, the coloring of $^2G$ is locally irregular.
\end{proof}

\section{Split graphs}\label{split_section}

We introduce the notion, and two theorems from \cite{Lintzmayer Mota Sambinelli} which we use to prove our main result for 2-multigraphs obtained from split graphs.
We say that a graph $G(X,Y)$ is \textit{split} if there exists a partition $\{X, Y\}$ of the vertex set of $G$ such that $G[X]$ is a complete graph and $Y$ is an independent set. Let $G(X,Y)$ be a split graph with $X = \{v_1, \ldots ,v_n\}$. For any vertex $v_i \in X$, we denote by $d_i$ the number of neighbors of $v_i$ in $Y$.

\begin{tw}[\cite{Lintzmayer Mota Sambinelli}]\label{split1}
Let $G(X,Y)$ be a split graph with $X = \{v_1, \ldots, v_n\}$ where $d_1 \geq \ldots \geq d_n$. If $n \geq 10$, then the following holds
\begin{itemize}
\item $\operatorname{lir}(G)\leq 2$ if and only if $d_1 \geq \lfloor \frac{n}{2} \rfloor$ or $d_2 \geq 1$;
\item $\operatorname{lir}(G) = 3$ if and only if $d_1 < \lfloor \frac{n}{2} \rfloor$ and $d_2 = 0$.
\end{itemize}
\end{tw}

\begin{tw}[\cite{Lintzmayer Mota Sambinelli}]\label{split2}
Let $G(X,Y)$ be a split graph with $X = \{v_1, \ldots, v_n\}$ where $d_1 \geq \ldots \geq d_n$. If $n \leq 9$, then the following holds
\begin{itemize}
\item $G$ is not decomposable if $G$ is isomorphic to $K_2$, $K_3$ or $P_4$;
\item If $G$ is decomposable, then
\begin{itemize}
\item If $d_1 > d_2 > \ldots > d_n$, then $\operatorname{lir}(G) = 1$;
\item  If $3 \leq n \leq 9$ and $\sum^n_{i=1} d_i \geq \lfloor \frac{n}{2} \rfloor$, then $\operatorname{lir}(G) = 2$;
\item If $8 \leq n \leq 9$, $\sum^3_{i=1} d_i=3$, and $d_2 \geq 1$, then $\operatorname{lir}(G) = 2$;
\item  If $n = 9$ and $d_1 = d_2 = 1$, then $\operatorname{lir}(G) = 2$;
\item For all the other cases, it follows that $\operatorname{lir}(G) = 3$. 
\end{itemize}
\end{itemize}
\end{tw}

\begin{tw}
If $G$ is a split graph different from $K_2$ then $\operatorname{lir}(^2G)\leq 2$.
\end{tw}
\begin{proof}
Let $G(X,Y)$ be a split graph with $X = \{v_1, \ldots, v_n\}$ where $d_1 \geq \ldots \geq d_n$. 
Without loss of generality, we may assume that the set of vertices $X$ which induces a clique is the maximum possible. 
If $n\leq2$ then $G$ is a tree, hence bipartite, and the result follows from Theorem \ref{cycle}. 

Next, assume that $n>2$. 
Due to Theorem~\ref{split1} and Theorem~\ref{split2}, only the following particular cases need to be considered:
\begin{itemize}
\item $d_1 < \lfloor \frac{n}{2} \rfloor$ and $d_2 = 0$, 
\item $d_1=d_2=1$ and $n\in \{6,7,8\}$;
\end{itemize}
these are the only split graphs with $n > 2$ and the locally irregular chromatic index equal to three.
For every other split graph $G$ with $n>2$, we have from Theorem~\ref{split1} and Theorem~\ref{split2} that $\operatorname{lir}(G)\leq 2$, and thus, $\operatorname{lir}(^2G) \leq 2$.

Consider the case when $d_1 < \lfloor \frac{n}{2} \rfloor$ and $d_2 = 0$.
From Theorem~\ref{cycle} we get a locally irregular 2-coloring of the complete subgraph of $G$ induced on $X$. 
Without loss of generality suppose that the blue degree of $v_1$ is the maximum possible. 
It is easy to see that the coloring of the remaining multiedges (pendant multiedges incident to $v_1$) blue-blue results in a locally irregular 2-coloring of $^2G$.

Now, we consider the case when $d_1=d_2=1$ and $n\in \{6,7,8\}$. 
Like in the previous case, from Theorem~\ref{cycle} we get a locally irregular 2-coloring of the 2-multigraph obtained from the complete subgraph of $G$ induced on $X$.
Note that the vertex with the largest blue degree is different from the vertex with the largest red degree.
Otherwise, all vertices would have the same red and blue degrees, which contradicts that the coloring is locally irregular.
Suppose that $v_1$ and $v_2$ have the largest blue and red degrees, respectively. Let $w_1$ be a neighbor of $v_1$ outside of the clique and $w_2$ be a neighbor of $v_2$ outside of the clique.
To finish the coloring, color the multiedge $v_1w_1$ blue-blue and  the multiedge $v_2w_2$ red-red.
It is easy to see that no conflict was created despite the possibility that $w_1=w_2$.
Hence, in this case, we also get $\operatorname{lir}(^2G) \leq 2$.
\end{proof}

\section{Planar graphs}\label{planar_section}

In~\cite{Bensmail Dross Nisse}, Bensmail, Dross, and Nisse proved that the locally irregular index of every colorable simple planar graph $G$ is at most 15 (14 if $|E(G)|$ is even). 
Naturally, this gives a bound of 15 on $\operatorname{lir}(^2G)$ for a simple planar colorable graph $G$. 
In this section, we significantly lower this bound.

\begin{tw}
If $G \neq K_2$ is a simple connected planar graph then $\operatorname{lir}(^2G)\leq 4$.
\end{tw}
\begin{proof}
From the Four-Color Theorem, we have that $\chi(G) \leq 4$.
Observe that, if $\chi(G) \leq 2$, i.e., $G$ is bipartite, the result follows from Theorem~\ref{cycle}.
Hence, in the following, we may assume that $\chi(G) \in \{3,4\}$.
Let $V_1,\dots, V_{\chi(G)}$ be the color classes of the proper vertex $\chi(G)$-coloring of $G$.


If $\chi(G) = 3$, denote by $G_1$ the spanning subgraph of $G$ with $E(G_1) = \{uv \colon u \in V_1\}$, and denote by $G_2$ the spanning subgraph of $G$ with $E(G_2) = E(G) \setminus E(G_1)$.
If $\chi(G) = 4$, let $G_1$ be the spanning subgraph of $G$ with $E(G_1) = \{uv \colon (u \in V_1 \cup V_3 \land v \in V_2) \lor (u \in V_3 \land v \in V_4)\}$, and let $G_2$ be the spanning subgraph of $G$ with $E(G_2) = E(G) \setminus E(G_1)$.
Clearly, $G_1$ and $G_2$ are bipartite graphs. 
Note also that, if $i,j \in \{1, \dots, \chi(G)\}$ and $i \neq j$ then there is at least one edge joining a vertex from $V_i$ with a vertex from $V_j$.
Hence, $\chi(G_1) = 2$ and $\chi(G_2) = 2$.

It might happen, however, that $G_1$ or $G_2$ contains an isolated edge, and thus, Theorem~\ref{cycle} cannot be applied to it. 
To avoid these problems, we use the following lemma proved by Baudon et al. (see Lemma 5.8. in~\cite{Baudon Bensmail Davot Hocquard}):
\begin{lem}[\cite{Baudon Bensmail Davot Hocquard}]\label{lemma_nice}
Assume that a graph $G$ without isolated edges and isolated triangles can be $2$-edge-colored with red and blue so that the induced red subgraph $G_R$ and the induced blue subgraph $G_B$ 
satisfy $\chi(G_R) = r$ and $\chi(G_B) = b$ where $r, b \geq 2$. 
Then $G$ can be $2$-edge-colored in such a way that $\chi(G_R) \leq r$, $\chi(G_B) \leq b$, and $G_R$ and $G_B$ are without isolated edges.
\end{lem}

From Lemma~\ref{lemma_nice}, we get a decomposition of $G$ into two bipartite graphs $G_1$ and $G_2$ without isolated edges. 
Hence, by Theorem~\ref{cycle}, both $^2G_1$ and $^2G_2$ admit locally irregular edge colorings using two colors; let those colors be red and blue in the case of $^2G_1$, and green and yellow in the case of $^2G_2$.
Subgraphs of $^2G_1$ and $^2G_2$ induced on edges of the same color form a decomposition of $^2G$ into four locally irregular submultigraphs.
Thus, $\operatorname{lir}(^2G) \leq 4$.      
\end{proof}

As an immediate consequence of the proof of the above theorem, we get the following:
\begin{cor}
For every connected tripartite or $4$-partite graph $G$, the multigraph $^2G$ satisfies $\operatorname{lir}(^2G)\leq 4$.
\end{cor}

\section{Subcubic graphs}\label{subcubic_section}

In this section, we present a new constant upper bound on $\operatorname{lir}(^2G)$ in the case when $G$ is a connected simple cubic graph different from $K_2$.
This new upper bound for subcubic graphs greatly improves the bound of 76 obtained from the result of Grzelec and Woźniak~\cite{Grzelec wozniak} for general 2-multigraphs. Moreover, we prove Conjecture~\ref{main} for two particular families of subcubic graphs. 

We start by introducing a few definitions and results that are useful to prove the above-mentioned new results.
Let $K''_{1,3}$ be the star $K_{1,3}$ with two edges subdivided once. 
An edge-decomposition of a connected graph is called \textit{pertinent} if it consists of paths of length two and at most one element isomorphic either to $K_{1,3}$  or $K''_{1,3}$. 
An edge-decomposition $\mathcal{D}$ of a graph is \textit{strongly pertinent} if it is pertinent, and when $\mathcal{D}$  contains an element isomorphic to $K''_{1,3}$, the graph has no pertinent edge-decomposition without $K''_{1,3}$. 
Naturally, two elements of a pertinent decomposition are \textit{adjacent}, if they share a common vertex.

In \cite{Bensmail Merker Thomassen} Bensmail,  Merker, and Thomassen showed that every simple colorable graph has a pertinent edge-decomposition; 
using this result Lužar, Przybyło, and Soták proved the following:
\begin{tw}[\cite{Luzar Przybylo Sotak}]\label{subcubic graphs}
Let $G$ be a simple colorable subcubic graph and $\mathcal{D}$ be a strongly pertinent edge-decomposition of $G$. Then, $G$ admits a locally irregular coloring with at most four colors such that
\begin{itemize}
\item the edges of every element of $\mathcal{D}$ are colored with the same color; and
\item if the edges of two adjacent elements $p_1$, $p_2$ of $\mathcal{D}$ are colored with the same color, then their common vertex is the central vertex of either $p_1$ or $p_2$.
\end{itemize}
\end{tw}

As an immediate consequence of the above result, we get that for every simple connected colorable subcubic graph $G$ which is not isomorphic to $K_2$, the multigraph $^2G$ satisfies $\operatorname{lir}(^2G)\leq 4$. 
We improve this corollary of Theorem~\ref{subcubic graphs} by proving the following theorem:
\begin{tw}
If $G \neq K_2$ is a simple connected subcubic graph then \linebreak $\operatorname{lir}(^2G)\leq 3$.
\end{tw}
\begin{proof}
It follows from Theorem \ref{cycle} that $\operatorname{lir}(^2G) \leq 2$ for every cactus $G \neq K_2$.
Since all simple connected uncolorable graphs are cacti, we may, in the following, assume that $G$ is colorable.

From the result of Bensmail,  Merker, and Thomassen~\cite{Bensmail Merker Thomassen} we have a pertinent decomposition of $G$, hence, we have a strongly pertinent decomposition $\mathcal{D}$ of $G$.
Then, from Theorem~\ref{subcubic graphs} we have a locally irregular coloring $\varphi$ of $G$ using at most four colors with the additional properties.
Let the colors $\varphi$ uses be $1,2,3,4$.
Without loss of generality we assume that the number of edges of color 4 is the minimum possible and that $\varphi(p) = 1$ if $p \in \mathcal{D}$ is isomorphic to $K_{1,3}$ or $K_{1,3}''$.

In the following, we use $G[\varphi^{-1}(c)]$ to denote the subgraph of $G$ induced on edges of color $c$ in $\varphi$.

    \begin{figure}[h]
        \centering
        \begin{subfigure}{0.12\textwidth}
        \begin{center}
            \begin{tikzpicture}[scale=0.6]
	\begin{pgfonlayer}{nodelayer}
		\node [style={black_dot}] (0) at (-1, 0) {};
		\node [style={black_dot}] (1) at (0, 0) {};
		\node [style={black_dot}] (2) at (1, 0) {};
	\end{pgfonlayer}
	\begin{pgfonlayer}{edgelayer}
		\draw [style={normal_edge}] (0) to (1);
		\draw [style={normal_edge}] (1) to (2);
	\end{pgfonlayer}
\end{tikzpicture}
        \end{center}
            
            \caption{}
            \label{subfiga}
        \end{subfigure}
        \hfill
        \begin{subfigure}{0.12\textwidth}
            \begin{center}
                \begin{tikzpicture}[scale=0.6]
	\begin{pgfonlayer}{nodelayer}
		\node [style={black_dot}] (0) at (-1, 1) {};
		\node [style={black_dot}] (1) at (0, 1) {};
		\node [style={black_dot}] (2) at (1, 1) {};
		\node [style={black_dot}] (3) at (0, 0) {};
		\node [style={black_dot}] (4) at (0, -1) {};
	\end{pgfonlayer}
	\begin{pgfonlayer}{edgelayer}
		\draw [style={normal_edge}] (0) to (1);
		\draw [style={normal_edge}] (1) to (2);
		\draw [style={normal_edge}] (1) to (3);
		\draw [style={normal_edge}] (3) to (4);
	\end{pgfonlayer}
\end{tikzpicture}    
            \end{center}
            
            \caption{}
            \label{subfigb}
        \end{subfigure}
        \hfill
        \begin{subfigure}{0.12\textwidth}
            \begin{center}
                \begin{tikzpicture}[scale=0.6]
	\begin{pgfonlayer}{nodelayer}
		\node [style={black_dot}] (0) at (-1, 1) {};
		\node [style={black_dot}] (1) at (0, 1) {};
		\node [style={black_dot}] (2) at (1, 1) {};
		\node [style={black_dot}] (3) at (0, 0) {};
		\node [style={black_dot}] (4) at (0, -1) {};
		\node [style={black_dot}] (5) at (-1, -1) {};
		\node [style={black_dot}] (6) at (1, -1) {};
	\end{pgfonlayer}
	\begin{pgfonlayer}{edgelayer}
		\draw [style={normal_edge}] (0) to (1);
		\draw [style={normal_edge}] (1) to (2);
		\draw [style={normal_edge}] (1) to (3);
		\draw [style={normal_edge}] (3) to (4);
		\draw [style={normal_edge}] (5) to (4);
		\draw [style={normal_edge}] (4) to (6);
	\end{pgfonlayer}
\end{tikzpicture}    
            \end{center}
            
            \caption{}
            \label{subfigc}
        \end{subfigure}
        \hfill
        \begin{subfigure}{0.12\textwidth}
            \begin{center}
                \begin{tikzpicture}[scale=0.6]
	\begin{pgfonlayer}{nodelayer}
		\node [style={black_dot}] (0) at (0, 0) {};
		\node [style={black_dot}] (1) at (0, 1) {};
		\node [style={black_dot}] (2) at (-0.75, -0.75) {};
		\node [style={black_dot}] (3) at (0.75, -0.75) {};
	\end{pgfonlayer}
	\begin{pgfonlayer}{edgelayer}
		\draw [style={normal_edge}] (1) to (0);
		\draw [style={normal_edge}] (0) to (2);
		\draw [style={normal_edge}] (0) to (3);
	\end{pgfonlayer}
\end{tikzpicture}    
            \end{center}
            
            \caption{}
            \label{subfigd}
        \end{subfigure}
        \hfill
        \begin{subfigure}{0.12\textwidth}
            \begin{center}
                \begin{tikzpicture}[scale=0.6]
	\begin{pgfonlayer}{nodelayer}
		\node [style={black_dot}] (0) at (0, 0) {};
		\node [style={black_dot}] (1) at (0, 1) {};
		\node [style={black_dot}] (2) at (-0.75, -0.75) {};
		\node [style={black_dot}] (3) at (0.75, -0.75) {};
		\node [style={black_dot}] (4) at (-1.5, -1.5) {};
		\node [style={black_dot}] (5) at (1.5, -1.5) {};
	\end{pgfonlayer}
	\begin{pgfonlayer}{edgelayer}
		\draw [style={normal_edge}] (1) to (0);
		\draw [style={normal_edge}] (0) to (2);
		\draw [style={normal_edge}] (0) to (3);
		\draw [style={normal_edge}] (4) to (2);
		\draw [style={normal_edge}] (3) to (5);
	\end{pgfonlayer}
\end{tikzpicture}    
            \end{center}
            
            \caption{}
            \label{subfige}
        \end{subfigure}
        \hfill  
        \begin{subfigure}{0.18\textwidth}
            \begin{center}
                \begin{tikzpicture}[scale=0.6]
	\begin{pgfonlayer}{nodelayer}
		\node [style={black_dot}] (0) at (0, 0) {};
		\node [style={black_dot}] (1) at (0, 1) {};
		\node [style={black_dot}] (2) at (-0.75, -0.75) {};
		\node [style={black_dot}] (3) at (0.75, -0.75) {};
		\node [style={black_dot}] (4) at (-1.5, -1.5) {};
		\node [style={black_dot}] (5) at (1.5, -1.5) {};
		\node [style={black_dot}] (6) at (1, -2.5) {};
		\node [style={black_dot}] (7) at (2.5, -1) {};
	\end{pgfonlayer}
	\begin{pgfonlayer}{edgelayer}
		\draw [style={normal_edge}] (0) to (2);
		\draw [style={normal_edge}] (4) to (2);
		\draw [style={normal_edge}] (3) to (5);
		\draw [style={normal_edge}] (6) to (5);
		\draw [style={normal_edge}] (5) to (7);
		\draw [style={dotted_edge}] (1) to (0);
		\draw [style={dotted_edge}] (0) to (3);
	\end{pgfonlayer}
\end{tikzpicture}    
            \end{center}
            
            \caption{}
            \label{subfigf}
        \end{subfigure}
        \hfill
        \begin{subfigure}{0.18\textwidth}
            \begin{center}
                \begin{tikzpicture}[scale=0.6]
	\begin{pgfonlayer}{nodelayer}
		\node [style={black_dot}] (0) at (0, 0) {};
		\node [style={black_dot}] (1) at (0, 1) {};
		\node [style={black_dot}] (2) at (-0.75, -0.75) {};
		\node [style={black_dot}] (3) at (0.75, -0.75) {};
		\node [style={black_dot}] (4) at (-1.5, -1.5) {};
		\node [style={black_dot}] (5) at (1.5, -1.5) {};
		\node [style={black_dot}] (6) at (1, -2.5) {};
		\node [style={black_dot}] (7) at (2.5, -1) {};
		\node [style={black_dot}] (8) at (-1, -2.5) {};
		\node [style={black_dot}] (9) at (-2.5, -1) {};
	\end{pgfonlayer}
	\begin{pgfonlayer}{edgelayer}
		\draw [style={normal_edge}] (0) to (2);
		\draw [style={normal_edge}] (4) to (2);
		\draw [style={normal_edge}] (3) to (5);
		\draw [style={normal_edge}] (6) to (5);
		\draw [style={normal_edge}] (5) to (7);
		\draw [style={dotted_edge}] (1) to (0);
		\draw [style={dotted_edge}] (0) to (3);
		\draw [style={dotted_edge}] (4) to (8);
		\draw [style={dotted_edge}] (4) to (9);
	\end{pgfonlayer}
\end{tikzpicture}    
            \end{center}
            
            \caption{}
            \label{subfigg}
        \end{subfigure}
                
        \caption{Possibilities for components of monochromatic subgraphs. If $\mathcal{D}$ is strongly pertinent, (f) and (g) are impossible to obtain (a decomposition without $K_{1,3}''$ is suggested using full and dotted edges).}
        \label{fig:monochromatic_components}
    \end{figure}
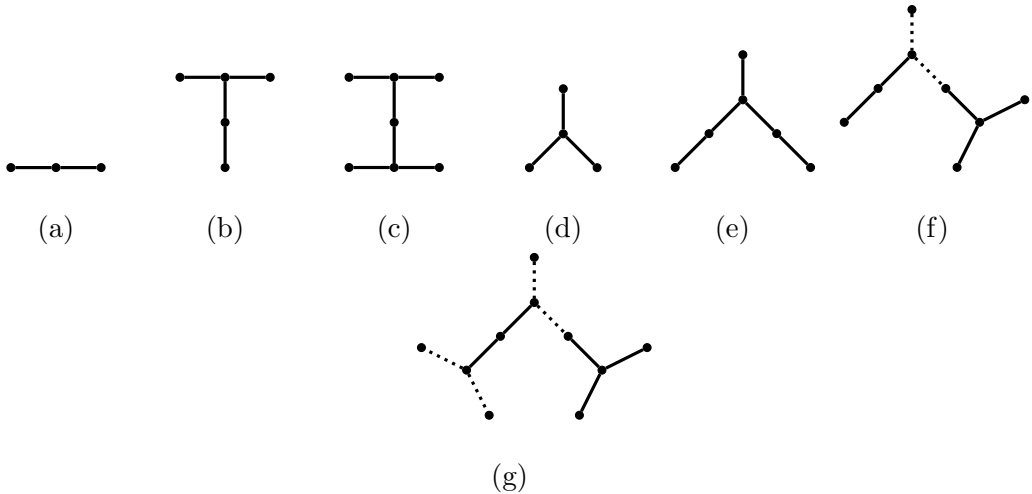

    The coloring $\varphi$ is locally irregular, and from the additional properties of $\varphi$ (see Theorem~\ref{subcubic graphs}) it follows that for every connected component of the monochromatic subgraph of $G$ we have, hypothetically speaking, seven possibilities how it could look like, see Figure~\ref{fig:monochromatic_components}.
    However, two of them, namely (f) and (g) in Figure~\ref{fig:monochromatic_components} contradicts the assumption that $\mathcal{D}$ is strongly pertinent.
    Hence, if the size of $G$ is even, every connected component of the monochromatic subgraph of $G$ is isomorphic to (a), (b), or (c) in Figure~\ref{fig:monochromatic_components}.
    If, on the other hand, the size of $G$ is odd, we have exactly one connected component of the subgraph of $G$ induced on color 1 which is isomorphic to (d) or (e) in Figure~\ref{fig:monochromatic_components}, and every other connected component of the monochromatic subgraph is isomorphic to (a), (b), or (c).

    \begin{figure}[h]
        \centering
        \begin{tikzpicture}
	\begin{pgfonlayer}{nodelayer}
		\node [style={black_dot}] (0) at (-1, 1) {};
		\node [style={black_dot}] (1) at (0, 1) {};
		\node [style={black_dot}] (2) at (1, 1) {};
		\node [style={black_dot}] (3) at (-1, -1) {};
		\node [style={black_dot}] (4) at (0, -1) {};
		\node [style={black_dot}] (5) at (1, -1) {};
		\node [style={black_dot}] (6) at (0, 0) {};
		\node [style={black_dot}] (7) at (1, 0) {};
		\node [style=empty, label={above:4}] (8) at (-0.5, 1) {};
		\node [style=empty, label={above:4}] (9) at (0.5, 1) {};
		\node [style=empty, label={left:4}] (10) at (0, 0.5) {};
		\node [style=empty, label={left:4}] (11) at (0, -0.5) {};
		\node [style=empty, label={below:4}] (12) at (-0.5, -1) {};
		\node [style=empty, label={below:4}] (13) at (0.5, -1) {};
		\node [style=empty, label={above:1}] (14) at (0.5, 0) {};
	\end{pgfonlayer}
	\begin{pgfonlayer}{edgelayer}
		\draw [style={normal_edge}] (0) to (1);
		\draw [style={normal_edge}] (1) to (2);
		\draw [style={normal_edge}] (1) to (6);
		\draw [style={normal_edge}] (6) to (4);
		\draw [style={normal_edge}] (4) to (3);
		\draw [style={normal_edge}] (4) to (5);
		\draw [style={dotted_edge}] (6) to (7);
	\end{pgfonlayer}
\end{tikzpicture}%
\hspace{2cm}%
\begin{tikzpicture}
	\begin{pgfonlayer}{nodelayer}
		\node [style={black_dot}] (0) at (-1, 1) {};
		\node [style={black_dot}] (1) at (0, 1) {};
		\node [style={black_dot}] (2) at (1, 1) {};
		\node [style={black_dot}] (3) at (-1, -1) {};
		\node [style={black_dot}] (4) at (0, -1) {};
		\node [style={black_dot}] (5) at (1, -1) {};
		\node [style={black_dot}] (6) at (0, 0) {};
		\node [style={black_dot}] (7) at (1, 0) {};
		\node [style=empty, label={above:4}] (8) at (-0.5, 1) {};
		\node [style=empty, label={above:4}] (9) at (0.5, 1) {};
		\node [style=empty, label={left:2}] (10) at (0, 0.5) {};
		\node [style=empty, label={left:2}] (11) at (0, -0.5) {};
		\node [style=empty, label={below:4}] (12) at (-0.5, -1) {};
		\node [style=empty, label={below:4}] (13) at (0.5, -1) {};
		\node [style=empty, label={above:1}] (14) at (0.5, 0) {};
	\end{pgfonlayer}
	\begin{pgfonlayer}{edgelayer}
		\draw [style={normal_edge}] (0) to (1);
		\draw [style={normal_edge}] (1) to (2);
		\draw [style={normal_edge}] (1) to (6);
		\draw [style={normal_edge}] (6) to (4);
		\draw [style={normal_edge}] (4) to (3);
		\draw [style={normal_edge}] (4) to (5);
		\draw [style={dotted_edge}] (6) to (7);
	\end{pgfonlayer}
\end{tikzpicture}
        \caption{Lowering the number of edges of color 4, if a component of $G[\varphi^{-1}(4)]$ is isomorphic to (c) (colors 1 and 2 may be changed to any combination of two colors from $\{1,2,3\}$).}
        \label{fig:avoid_H_shape}
    \end{figure}
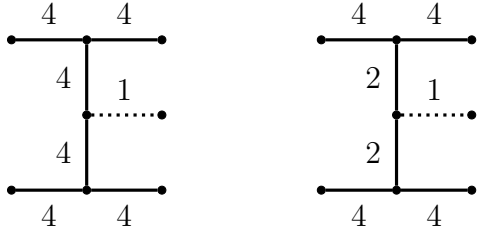
    
    From the assumption that the number of edges of color 4 is the minimum possible in $\varphi$, we get that there is no connected component of the subgraph of $G$ induced on edges of color 4 which is isomorphic to (c);
    if this was the case, we could recolor the edges of the middle part of (c) by a color that is not used on elements of $\mathcal{D}$ adjacent to it, see Figure~\ref{fig:avoid_H_shape}, and, thus, lower the number of edges colored by 4.
    Hence, connected components of $G[\varphi^{-1}(4)]$ are isomorphic to (a) or (b) in Figure~\ref{fig:monochromatic_components}.

    In the following, we will create a locally irregular coloring of $^2G$ which will use colors red, green, and blue, and which will have the following property (P): 
    \textit{if $p \in \mathcal{D}$ has $\varphi(p) \in \{1, 2, 3\}$ then its pendant vertices do not have red, green or blue degree equal to five in the coloring of $^2G$.}
    We color the edges of $^2G$ in two steps.
    Note that, if a color is assigned to an edge of $^2G$ in any step, it is never changed in the subsequent steps.
    
    First, we color every multiedge of $^2G$ which corresponds to an edge of color 1, 2 or 3 in $\varphi$ in the following way:
    if $\varphi(e) = 1$ then the multiedge $e$ is colored blue-blue in $^2G$
    if $\varphi(e) = 2$ then the multiedge $e$ is colored green-green in $^2G$,
    and if $\varphi(e) = 3$ then the multiedge $e$ is colored red-red in $^2G$. 
    Note that, after this step, multiedges corresponding to edges of color 4 in $G$ are not colored, 
    and the obtained partial coloring of $^2G$ has the property (P), because color degrees of every vertex in such a partial coloring are even.
    
    We now proceed to the second stage of the coloring of $^2G$ where the rest of the edges, namely $E(G[\varphi^{-1}(4)])$, are being colored.
    As we showed before, components of $G[\varphi^{-1}(4)]$ are isomorphic to (a) or (b).
    
    Suppose first that we have a component of $G[\varphi^{-1}(4)]$ which is isomorphic to (a); denote by $p$ the corresponding element of $\mathcal{D}$.
    If the central vertex of $p$ is not an end vertex of some $q \in \mathcal{D}$, color each multiedge of $p$ in $^2G$ red-blue.
    If the central vertex of $p$ is an end vertex of some $q \in \mathcal{D}$, we distinguish cases depending on $\varphi(q)$: 
    if $\varphi(q) = 1$ (i.e. the multiedges of $q$ are colored blue-blue in $^2G$), color each multiedge of $p$ green-red, 
    if $\varphi(q) = 2$ (i.e. the multiedges of $q$ are colored green-green in $^2G$), color each multiedge of $p$ blue-red,
    if $\varphi(q) = 3$ (i.e. the multiedges of $q$ are colored red-red in $^2G$), color each multiedge of $p$ green-blue.
    By doing this, we do not create a conflict between the central vertex of $p$ and the central vertex of $q$.
    Hence, if the conflict was created, it is between the end vertex of $p$ and some other vertex.
    
    Suppose that the multiedges of $p$ were colored green and blue (if other color combination was used, simply change the roles of the colors in the following argument).
    Let $v$ be an end vertex of $p$.
    Suppose first that $v$ is a central vertex of $r \in \mathcal{D}$.
    If $r$ is blue (a similar argument can be used for green) then the blue degree of $v$ is five. 
    The blue degree of the central vertex of $p$ is at most four, and the neighbors of $v$ in $r$ do not have blue degree five (one of them is in every case pendant and we use the property (P), and the second one is either pendant or the vertex of degree three in (d) or (e), and we use either the property (P) or the fact that it has blue degree six). 
    If $r$ is red then no conflict was created on this end of $p$.

    Suppose next that two elements $s$ and $t$ of $\mathcal{D}$ adjacent to $p$ contain $v$.
    We distinguish two non-symmetric cases depending on the colors of $s$ and $t$.
    Suppose that $\varphi(s) = 1$ and $\varphi(t) = 2$ (i.e., the multiedges of $s$ and $t$ are colored blue-blue and green-green, respectively).
    In this case, the blue and green degrees of the central vertex of $p$ are even, blue and green degrees of $v$ are three, the blue degree of the neighbor of $v$ on $s$ is at least four, and the green degree of the neighbor of $v$ on $t$ is at least four. 
    Hence, no conflicts were created, and the property (P) was not violated.
    Consider now the second case, when $\varphi(s) = 1$ and $\varphi(t) = 3$ (i.e., the multiedges of $s$ and $t$ are colored blue-blue and red-red, respectively). 
    In this case, the blue and green degrees of the central vertex of $p$ are even. Blue, green, and red degrees of $v$ are three, one, and two, respectively. 
    The blue degree of the neighbor of $v$ on $s$ is at least four, and the red degree of the neighbor of $v$ on $t$ is at least four.
    Hence, no conflict was created in this case also, and property (P) was not violated.
    For the overview of these cases, see Figure~\ref{fig:4_pendant_pst}.

    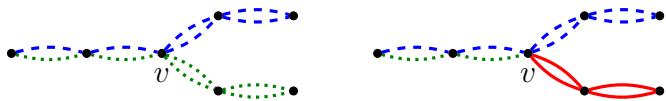
\begin{figure}[h]
        \centering
        \begin{tikzpicture}
	\begin{pgfonlayer}{nodelayer}
		\node [style={black_dot}] (0) at (-1, 0) {};
		\node [style={black_dot}] (1) at (0, 0) {};
		\node [style={black_dot}, label={below:$v$}] (2) at (1, 0) {};
		\node [style={black_dot}] (3) at (1.75, 0.5) {};
		\node [style={black_dot}] (4) at (1.75, -0.5) {};
		\node [style={black_dot}] (5) at (2.75, 0.5) {};
		\node [style={black_dot}] (6) at (2.75, -0.5) {};
	\end{pgfonlayer}
	\begin{pgfonlayer}{edgelayer}
		\draw [style={green_edge}, bend right=15] (0) to (1);
		\draw [style={green_edge}, bend right=15] (1) to (2);
		\draw [style={blue_edge}, bend left=15] (0) to (1);
		\draw [style={blue_edge}, bend left=15] (1) to (2);
		\draw [style={blue_edge}, bend left=15] (2) to (3);
		\draw [style={green_edge}, bend right=15] (2) to (4);
		\draw [style={blue_edge}, bend right=15] (2) to (3);
		\draw [style={green_edge}, bend left=15] (2) to (4);
		\draw [style={blue_edge}, bend right=15] (3) to (5);
		\draw [style={green_edge}, bend left=15] (4) to (6);
		\draw [style={blue_edge}, bend left=15] (3) to (5);
		\draw [style={green_edge}, bend right=15] (4) to (6);
	\end{pgfonlayer}
\end{tikzpicture}%
\hspace{1cm}%
\begin{tikzpicture}
	\begin{pgfonlayer}{nodelayer}
		\node [style={black_dot}] (0) at (-1, 0) {};
		\node [style={black_dot}] (1) at (0, 0) {};
		\node [style={black_dot}, label={below:$v$}] (2) at (1, 0) {};
		\node [style={black_dot}] (3) at (1.75, 0.5) {};
		\node [style={black_dot}] (4) at (1.75, -0.5) {};
		\node [style={black_dot}] (5) at (2.75, 0.5) {};
		\node [style={black_dot}] (6) at (2.75, -0.5) {};
	\end{pgfonlayer}
	\begin{pgfonlayer}{edgelayer}
		\draw [style={green_edge}, bend right=15] (0) to (1);
		\draw [style={green_edge}, bend right=15] (1) to (2);
		\draw [style={blue_edge}, bend left=15] (0) to (1);
		\draw [style={blue_edge}, bend left=15] (1) to (2);
		\draw [style={blue_edge}, bend left=15] (2) to (3);
		\draw [style={red_edge}, bend right=15] (2) to (4);
		\draw [style={blue_edge}, bend right=15] (2) to (3);
		\draw [style={red_edge}, bend left=15] (2) to (4);
		\draw [style={blue_edge}, bend right=15] (3) to (5);
		\draw [style={red_edge}, bend left=15] (4) to (6);
		\draw [style={blue_edge}, bend left=15] (3) to (5);
		\draw [style={red_edge}, bend right=15] (4) to (6);
	\end{pgfonlayer}
\end{tikzpicture}
        \caption{Cases when $p \in \mathcal{D}$ of color 4 in $\varphi$ is isomorphic to (a), and the end vertex $v$ of $p$ is an end vertex of $s$ and $t$, where $s,t \in \mathcal{D}$.}
        \label{fig:4_pendant_pst}
    \end{figure}   

    Suppose therefore that we colored every uncolored multiedges that corresponded to edges of components of $G[\varphi^{-1}(4)]$ isomorphic to $P_3$.
    We proceed to the coloring of the remaining multiedges, which correspond to edges of the components of $G[\varphi^{-1}(4)]$ isomorphic to (b).
    Let $p_1,p_2 \in \mathcal{D}$ be two adjacent elements of color 4 in $\varphi$. 
    Without loss of generality suppose that the central vertex of $p_1$ is the pendant vertex of $p_2$.
    Observe that both $p_1$ and $p_2$ are adjacent to elements of $\mathcal{D}$ of all colors from $\{1,2,3\}$, as otherwise we could recolor $p_1$ or $p_2$ by a color from $\{1,2,3\}$ and obtain a suitable locally irregular coloring of $G$ with fewer edges of color 4. 
    Hence, there are some elements $r_1, r_2, r_3$ of different colors from $\{1,2,3\}$ such that the pendant vertex of $r_1$ is the central vertex of $p_2$, and a pendant vertex of $p_2$ is also a pendant vertex of $r_2$ and $r_3$.
    Denote by $x$ the central vertex of $p_2$ and by $y$ the pendant vertex of $p_2$, $r_2$ and $r_3$.
    Observe that the neighbor of $x$ on $r_1$ is incident to three edges of color $\varphi(r_1)$; if this was not the case, then $p_1$ could be colored by $\varphi(r_1)$, and the number of edges colored by $4$ would be lower (and in some cases, the property that $\mathcal{D}$ is strongly pertinent would be contradicted, see (e) and (f) in Figure~\ref{fig:monochromatic_components}).
    
    For the number of elements of $\mathcal{D}$ adjacent to $p_1$, we have two possibilities, and we distinguish the cases.

    \textbf{Case 1.} $p_1$ is adjacent to four elements of $\mathcal{D}$, namely $p_2$, $s_1$, $s_2$ and $s_3$, where a pendant vertex of $p_1$ is the central vertex of $s_1$, and the other pendant vertex of $p_1$ is a pendant vertex of both $s_2$ and $s_3$. 
    For the overview of this case, see Figure~\ref{fig:case1}.
    Suppose that $\varphi(s_1) = 1$, $\varphi(s_2) = 2$, and $\varphi(s_3) = 3$, i.e., the multiedges corresponding to $s_1$, $s_2$ and $s_3$ are colored blue-blue, green-green and red-red, respectively (the similar argument works when the colors of $s_1, s_2, s_3$ are permuted).
    Color each multiedge of $p_1$ blue-red, and each multiedge of $p_2$ blue-green.
    See Figure~\ref{fig:case1} for an overview of color degrees. 
    This yields that such a coloring does not produce a conflict. 
    Note that each neighbor of the central vertex of $p_1$ on $s_1$ has the blue degree different from five, as it is either a pendant vertex of $s_1$ and (P) holds, or it is a vertex of degree three in (e).

    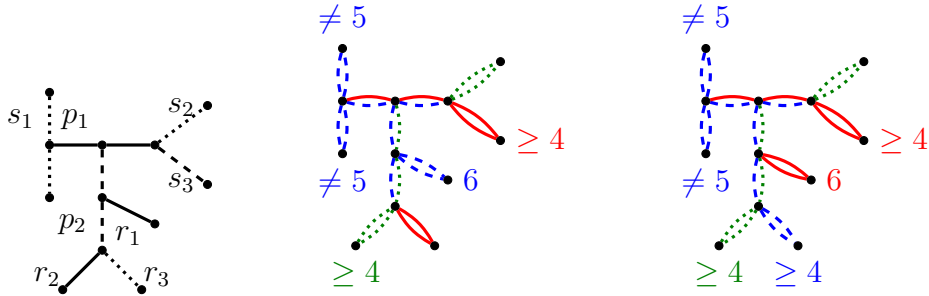
\begin{figure}[h]
        \centering
        \begin{tikzpicture}[scale=0.7]
	\begin{pgfonlayer}{nodelayer}
		\node [style={black_dot}] (0) at (-1, 0) {};
		\node [style={black_dot}] (1) at (0, 0) {};
		\node [style={black_dot}] (2) at (1, 0) {};
		\node [style={black_dot}] (3) at (-1, 1) {};
		\node [style={black_dot}] (4) at (-1, -1) {};
		\node [style={black_dot}] (5) at (2, 0.75) {};
		\node [style={black_dot}] (8) at (2, -0.75) {};
		\node [style={black_dot}] (10) at (0, -1) {};
		\node [style={black_dot}] (11) at (0, -2) {};
		\node [style={black_dot}] (12) at (1, -1.5) {};
		\node [style={black_dot}] (13) at (-0.75, -2.75) {};
		\node [style={black_dot}] (14) at (0.75, -2.75) {};
		\node [style=empty, label={left:$s_1$}] (15) at (-1, 0.5) {};
		\node [style=empty, label={above:$p_1$}] (16) at (-0.5, 0) {};
		\node [style=empty, label={left:$p_2$}] (17) at (0, -1.5) {};
		\node [style=empty, label={above:$s_2$}] (18) at (1.5, 0.25) {};
		\node [style=empty, label={below:$s_3$}] (19) at (1.5, -0.25) {};
		\node [style=empty, label={below:$r_1$}] (20) at (0.5, -1.25) {};
		\node [style=empty, label={left:$r_2$}] (21) at (-0.5, -2.5) {};
		\node [style=empty, label={right:$r_3$}] (22) at (0.5, -2.5) {};
	\end{pgfonlayer}
	\begin{pgfonlayer}{edgelayer}
		\draw [style={normal_edge}] (0) to (1);
		\draw [style={normal_edge}] (1) to (2);
		\draw [style={dotted_edge}] (3) to (0);
		\draw [style={dotted_edge}] (0) to (4);
		\draw [style={dotted_edge}] (2) to (5);
		\draw [style={dashed_edge}] (2) to (8);
		\draw [style={dashed_edge}] (1) to (10);
		\draw [style={dashed_edge}] (10) to (11);
		\draw [style={normal_edge}] (10) to (12);
		\draw [style={normal_edge}] (13) to (11);
		\draw [style={dotted_edge}] (11) to (14);
	\end{pgfonlayer}
\end{tikzpicture}
\hspace{1cm}
\begin{tikzpicture}[scale=0.7]
	\begin{pgfonlayer}{nodelayer}
		\node [style={black_dot}] (0) at (0.5, 0) {};
		\node [style={black_dot}] (1) at (1.5, 0) {};
		\node [style={black_dot}] (2) at (2.5, 0) {};
		\node [style={black_dot}, label={above:$\color{blue}{\neq 5}$}] (3) at (0.5, 1) {};
		\node [style={black_dot}, label={below:$\color{blue}{\neq 5}$}] (4) at (0.5, -1) {};
		\node [style={black_dot}] (5) at (3.5, 0.75) {};
		\node [style={black_dot}, label={right:$\color{red}{\geq 4}$}] (6) at (3.5, -0.75) {};
		\node [style={black_dot}] (7) at (1.5, -1) {};
		\node [style={black_dot}] (8) at (1.5, -2) {};
		\node [style={black_dot}, label={right:$\color{blue}{6}$}] (9) at (2.5, -1.5) {};
		\node [style={black_dot}, label={below:$\color{Green}{\geq 4}$}] (10) at (0.75, -2.75) {};
		\node [style={black_dot}] (11) at (2.25, -2.75) {};
	\end{pgfonlayer}
	\begin{pgfonlayer}{edgelayer}
		\draw [style={blue_edge}, bend right=15] (3) to (0);
		\draw [style={blue_edge}, bend right=15] (0) to (4);
		\draw [style={green_edge}, bend right=15] (2) to (5);
		\draw [style={red_edge}, bend right=15] (2) to (6);
		\draw [style={blue_edge}, bend right=15] (7) to (9);
		\draw [style={green_edge}, bend left=15] (10) to (8);
		\draw [style={red_edge}, bend right=15] (8) to (11);
		\draw [style={blue_edge}, bend right=15] (4) to (0);
		\draw [style={blue_edge}, bend right=15] (0) to (3);
		\draw [style={blue_edge}, bend left=15] (7) to (9);
		\draw [style={red_edge}, bend right=15] (6) to (2);
		\draw [style={green_edge}, bend left=15] (2) to (5);
		\draw [style={green_edge}, bend left=15] (8) to (10);
		\draw [style={red_edge}, bend left=15] (8) to (11);
		\draw [style={blue_edge}, bend right=15] (0) to (1);
		\draw [style={blue_edge}, bend right=15] (1) to (2);
		\draw [style={blue_edge}, bend right=15] (1) to (7);
		\draw [style={blue_edge}, bend right=15] (7) to (8);
		\draw [style={red_edge}, bend left=15] (0) to (1);
		\draw [style={red_edge}, bend left=15] (1) to (2);
		\draw [style={green_edge}, bend left=15] (1) to (7);
		\draw [style={green_edge}, bend left=15] (7) to (8);
	\end{pgfonlayer}
\end{tikzpicture}
\hspace{1cm}
\begin{tikzpicture}[scale=0.7]
	\begin{pgfonlayer}{nodelayer}
		\node [style={black_dot}] (0) at (0.5, 0) {};
		\node [style={black_dot}] (1) at (1.5, 0) {};
		\node [style={black_dot}] (2) at (2.5, 0) {};
		\node [style={black_dot}, label={above:$\color{blue}{\neq 5}$}] (3) at (0.5, 1) {};
		\node [style={black_dot}, label={below:$\color{blue}{\neq 5}$}] (4) at (0.5, -1) {};
		\node [style={black_dot}] (5) at (3.5, 0.75) {};
		\node [style={black_dot}, label={right:$\color{red}{\geq 4}$}] (6) at (3.5, -0.75) {};
		\node [style={black_dot}] (7) at (1.5, -1) {};
		\node [style={black_dot}] (8) at (1.5, -2) {};
		\node [style={black_dot}, label={right:$\color{red}{6}$}] (9) at (2.5, -1.5) {};
		\node [style={black_dot}, label={below:$\color{Green}{\geq 4}$}] (10) at (0.75, -2.75) {};
		\node [style={black_dot}, label={below:$\color{blue}{\geq 4}$}] (11) at (2.25, -2.75) {};
	\end{pgfonlayer}
	\begin{pgfonlayer}{edgelayer}
		\draw [style={blue_edge}, bend right=15] (3) to (0);
		\draw [style={blue_edge}, bend right=15] (0) to (4);
		\draw [style={green_edge}, bend right=15] (2) to (5);
		\draw [style={red_edge}, bend right=15] (2) to (6);
		\draw [style={red_edge}, bend right=15] (7) to (9);
		\draw [style={green_edge}, bend left=15] (10) to (8);
		\draw [style={blue_edge}, bend right=15] (8) to (11);
		\draw [style={blue_edge}, bend right=15] (4) to (0);
		\draw [style={blue_edge}, bend right=15] (0) to (3);
		\draw [style={red_edge}, bend left=15] (7) to (9);
		\draw [style={red_edge}, bend right=15] (6) to (2);
		\draw [style={green_edge}, bend left=15] (2) to (5);
		\draw [style={green_edge}, bend left=15] (8) to (10);
		\draw [style={blue_edge}, bend left=15] (8) to (11);
		\draw [style={blue_edge}, bend right=15] (0) to (1);
		\draw [style={blue_edge}, bend right=15] (1) to (2);
		\draw [style={blue_edge}, bend right=15] (1) to (7);
		\draw [style={blue_edge}, bend right=15] (7) to (8);
		\draw [style={red_edge}, bend left=15] (0) to (1);
		\draw [style={red_edge}, bend left=15] (1) to (2);
		\draw [style={green_edge}, bend left=15] (1) to (7);
		\draw [style={green_edge}, bend left=15] (7) to (8);
	\end{pgfonlayer}
\end{tikzpicture}
        \caption{Elements of $\mathcal{D}$ adjacent to $p_1$ and $p_2$ and possible colorings of $p_1$ and $p_2$ from Case 1.}
        \label{fig:case1}
    \end{figure}

    \textbf{Case 2.} $p_1$ is adjacent to five elements of $\mathcal{D}$, namely $p_2$, $t_1$, $t_2$, $t_3$ and $t_4$, where $p_1$, $t_1$ and $t_2$ share a common vertex, and $p_1$, $t_3$ and $t_4$ share a common vertex. 
    Suppose that $\varphi(t_1) = \varphi(t_4) = 1$, $\varphi(t_2) = 2$ and $\varphi(t_3) = 3$ (the similar argument works when the colors of $t_1, t_2, t_3, t_4$ are permuted). 
    There are two subcases to consider, depending on the color of $r_1$.
    However, despite the color of $r_1$, we color the edges of $p_1$ and $p_2$ in the same way: color the multiedge incident to a common vertex of $p_1$, $t_1$ and $t_2$ red-red, and the other multiedge of $p_1$ green-green. 
    Then color the multiedges of $p_2$ green-red.
    It is easy to see from the color degrees of the vertices of $p_1$, $p_2$ and their neighbors, that such a coloring does not create any conflict, see Figure~\ref{fig:case2}.

    \begin{figure}[h]
        \centering
        \begin{tikzpicture}[scale=0.7]
	\begin{pgfonlayer}{nodelayer}
		\node [style={black_dot}] (0) at (-1, 0) {};
		\node [style={black_dot}] (1) at (0, 0) {};
		\node [style={black_dot}] (2) at (1, 0) {};
		\node [style={black_dot}] (3) at (-2, 0.75) {};
		\node [style={black_dot}] (4) at (-2, -0.75) {};
		\node [style={black_dot}] (5) at (2, 0.75) {};
		\node [style={black_dot}] (8) at (2, -0.75) {};
		\node [style={black_dot}] (10) at (0, -1) {};
		\node [style={black_dot}] (11) at (0, -2) {};
		\node [style={black_dot}] (12) at (1, -1.5) {};
		\node [style={black_dot}] (13) at (-0.75, -2.75) {};
		\node [style={black_dot}] (14) at (0.75, -2.75) {};
		\node [style=empty, label={above:{$t_1$}}] (15) at (-1.5, 0.25) {};
		\node [style=empty, label={above:{$p_1$}}] (16) at (-0.5, 0) {};
		\node [style=empty, label={left:{$p_2$}}] (17) at (0, -1.5) {};
		\node [style=empty, label={above:{$t_3$}}] (18) at (1.5, 0.25) {};
		\node [style=empty, label={below:{$t_4$}}] (19) at (1.5, -0.25) {};
		\node [style=empty, label={above:{$r_1$}}] (20) at (0.5, -1.25) {};
		\node [style=empty, label={left:{$r_2$}}] (21) at (-0.5, -2.5) {};
		\node [style=empty, label={right:{$r_3$}}] (22) at (0.5, -2.5) {};
		\node [style=empty, label={below:{$t_2$}}] (23) at (-1.5, -0.25) {};
	\end{pgfonlayer}
	\begin{pgfonlayer}{edgelayer}
		\draw [style={normal_edge}] (0) to (1);
		\draw [style={normal_edge}] (1) to (2);
		\draw [style={dotted_edge}] (3) to (0);
		\draw [style={dotted_edge}] (0) to (4);
		\draw [style={dotted_edge}] (2) to (5);
		\draw [style={dashed_edge}] (2) to (8);
		\draw [style={dashed_edge}] (1) to (10);
		\draw [style={dashed_edge}] (10) to (11);
		\draw [style={normal_edge}] (10) to (12);
		\draw [style={normal_edge}] (13) to (11);
		\draw [style={dotted_edge}] (11) to (14);
	\end{pgfonlayer}
\end{tikzpicture}
\hspace{1cm}
\begin{tikzpicture}[scale=0.7]
	\begin{pgfonlayer}{nodelayer}
		\node [style={black_dot}] (0) at (0.5, 0) {};
		\node [style={black_dot}] (1) at (1.5, 0) {};
		\node [style={black_dot}] (2) at (2.5, 0) {};
		\node [style={black_dot}, label={above:$\color{blue}{\geq 4}$}] (3) at (-0.5, 0.75) {};
		\node [style={black_dot}, label={below:$\color{Green}{\geq 4}$}] (4) at (-0.5, -0.75) {};
		\node [style={black_dot}, label={above:$\color{red}{\geq 4}$}] (5) at (3.5, 0.75) {};
		\node [style={black_dot}, label={below:$\color{blue}{\geq 4}$}] (6) at (3.5, -0.75) {};
		\node [style={black_dot}] (7) at (1.5, -1) {};
		\node [style={black_dot}] (8) at (1.5, -2) {};
		\node [style={black_dot}, label={below:$\color{blue}{6}$}] (9) at (2.5, -1.5) {};
		\node [style={black_dot}, label={below:$\color{Green}{\geq 4}$}] (10) at (0.75, -2.75) {};
		\node [style={black_dot}, label={below:$\color{red}{\geq 4}$}] (11) at (2.25, -2.75) {};
	\end{pgfonlayer}
	\begin{pgfonlayer}{edgelayer}
		\draw [style={blue_edge}, bend right=15] (3) to (0);
		\draw [style={green_edge}, bend right=15] (0) to (4);
		\draw [style={red_edge}, bend right=15] (2) to (5);
		\draw [style={blue_edge}, bend right=15] (2) to (6);
		\draw [style={blue_edge}, bend right=15] (7) to (9);
		\draw [style={green_edge}, bend left=15] (10) to (8);
		\draw [style={red_edge}, bend right=15] (8) to (11);
		\draw [style={green_edge}, bend right=15] (4) to (0);
		\draw [style={blue_edge}, bend right=15] (0) to (3);
		\draw [style={blue_edge}, bend left=15] (7) to (9);
		\draw [style={blue_edge}, bend right=15] (6) to (2);
		\draw [style={red_edge}, bend left=15] (2) to (5);
		\draw [style={green_edge}, bend left=15] (8) to (10);
		\draw [style={red_edge}, bend left=15] (8) to (11);
		\draw [style={red_edge}, bend right=15] (0) to (1);
		\draw [style={green_edge}, bend right=15] (1) to (2);
		\draw [style={red_edge}, bend right=15] (1) to (7);
		\draw [style={red_edge}, bend right=15] (7) to (8);
		\draw [style={red_edge}, bend left=15] (0) to (1);
		\draw [style={green_edge}, bend left=15] (1) to (2);
		\draw [style={green_edge}, bend left=15] (1) to (7);
		\draw [style={green_edge}, bend left=15] (7) to (8);
	\end{pgfonlayer}
\end{tikzpicture}
\hspace{1cm}
\begin{tikzpicture}[scale=0.7]
	\begin{pgfonlayer}{nodelayer}
		\node [style={black_dot}] (0) at (0.5, 0) {};
		\node [style={black_dot}] (1) at (1.5, 0) {};
		\node [style={black_dot}] (2) at (2.5, 0) {};
		\node [style={black_dot}, label={above:$\color{blue}{\geq 4}$}] (3) at (-0.5, 0.75) {};
		\node [style={black_dot}, label={below:$\color{Green}{\geq 4}$}] (4) at (-0.5, -0.75) {};
		\node [style={black_dot}, label={above:$\color{red}{\geq 4}$}] (5) at (3.5, 0.75) {};
		\node [style={black_dot}, label={below:$\color{blue}{\geq 4}$}] (6) at (3.5, -0.75) {};
		\node [style={black_dot}] (7) at (1.5, -1) {};
		\node [style={black_dot}] (8) at (1.5, -2) {};
		\node [style={black_dot}, label={below:$\color{red}{6}$}] (9) at (2.5, -1.5) {};
		\node [style={black_dot}, label={below:$\color{Green}{\geq 4}$}] (10) at (0.75, -2.75) {};
		\node [style={black_dot}, label={below:$\color{blue}{\geq 4}$}] (11) at (2.25, -2.75) {};
	\end{pgfonlayer}
	\begin{pgfonlayer}{edgelayer}
		\draw [style={blue_edge}, bend right=15] (3) to (0);
		\draw [style={green_edge}, bend right=15] (0) to (4);
		\draw [style={red_edge}, bend right=15] (2) to (5);
		\draw [style={blue_edge}, bend right=15] (2) to (6);
		\draw [style={red_edge}, bend right=15] (7) to (9);
		\draw [style={green_edge}, bend left=15] (10) to (8);
		\draw [style={blue_edge}, bend right=15] (8) to (11);
		\draw [style={green_edge}, bend right=15] (4) to (0);
		\draw [style={blue_edge}, bend right=15] (0) to (3);
		\draw [style={red_edge}, bend left=15] (7) to (9);
		\draw [style={blue_edge}, bend right=15] (6) to (2);
		\draw [style={red_edge}, bend left=15] (2) to (5);
		\draw [style={green_edge}, bend left=15] (8) to (10);
		\draw [style={blue_edge}, bend left=15] (8) to (11);
		\draw [style={red_edge}, bend right=15] (0) to (1);
		\draw [style={green_edge}, bend right=15] (1) to (2);
		\draw [style={red_edge}, bend right=15] (1) to (7);
		\draw [style={red_edge}, bend right=15] (7) to (8);
		\draw [style={red_edge}, bend left=15] (0) to (1);
		\draw [style={green_edge}, bend left=15] (1) to (2);
		\draw [style={green_edge}, bend left=15] (1) to (7);
		\draw [style={green_edge}, bend left=15] (7) to (8);
	\end{pgfonlayer}
\end{tikzpicture}
        \caption{Elements of $\mathcal{D}$ adjacent to $p_1$ and $p_2$ and possible colorings of $p_1$ and $p_2$ from Case 2 (the subcase when $r_1$ is blue is in the middle, and the subcase when $r_1$ is red is on the right).}
        \label{fig:case2}
    \end{figure}
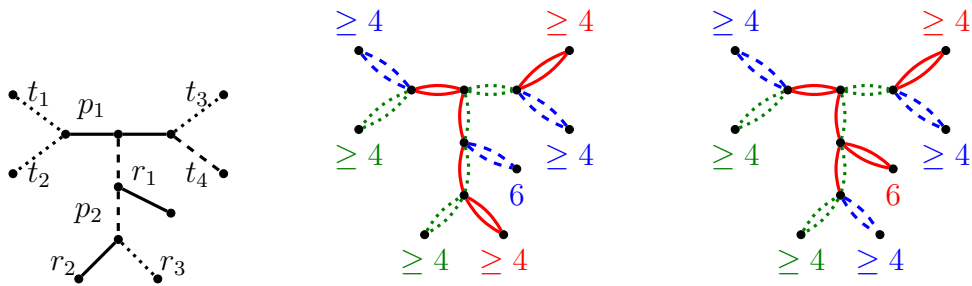

    Hence, every edge of $^2G$ was colored and no conflict was created in the process, which completes the proof.
\end{proof}

Now, we prove Conjecture 3 for two special families of subcubic graphs.

\begin{tw}\label{subcubic_independent}
    Let $G$ be a connected subcubic graph different from $K_2$. 
    If the set of all vertices of degree three in $G$ is an independent set then $\operatorname{lir}(^2G) \leq 2$. 
\end{tw}
\begin{proof}
    We prove that there is a locally irregular edge 2-coloring of $^2G$ with colors red and blue such that no pendant multiedge colored red-blue has an end vertex with the red and the blue degree equal to three. 
    Suppose to the contrary that this is not true.
    Let $G$ be a counterexample with the minimum number of edges.

    Note that 2-multigraphs created from paths (of length at least two) and cycles admit locally irregular edge 2-colorings, see Theorem~\ref{cycle}. 
    Moreover, there is no vertex of degree six in these 2-multigraphs, hence, there is not a pendant multiedge colored red-blue with an end vertex of the red and blue degrees three and three. Thus, $G$ is neither a path nor a cycle.
    
    Furthermore, $G$ is neither $K_{1,3}$ nor $K_{1,3}$ with some of its edges subdivided once since these graphs are locally irregular.
    If $G$ is $K_{1,3}$ with one of its edges subdivided twice, and the remaining two edges subdivided once, color the last two edges on the longest path blue, and the rest of the graph red; this yields $\operatorname{lir}(G) \leq 2$ and subsequently $\operatorname{lir}(^2G) \leq 2$.
    If $G$ is a triangle with a pendant edge attached to one of its vertices, we can color two adjacent edges blue, and two remaining edges red; we get $\operatorname{lir}(G) \leq2$, and thus $\operatorname{lir}(^2G) \leq 2$.

    Now, suppose that $H$ is a subgraph of $G$ isomorphic to $K_{1,3}$, or $K_{1,3}$ with at most two of its edges subdivided once, attached to the rest of the graph $G$ by a vertex of degree one adjacent to a vertex of degree three in $H$.
    We already showed that $G$ is not $K_{1,3}$ with all its edges subdivided once.
    Hence, we may suppose that the graph $G'=G-H$ is different from $K_2$.
    From the minimality of $G$ it follows that there is a locally irregular edge 2-coloring of $^2G'$ in which red-blue colored pendant multiedges are not incident to a vertex of the red and blue degrees three and three.
    Since $H$ is attached to a pendant edge of $G'$ at least one of the two monochromatic colorings of $^2H$ does not create a conflict on the ends of the pendant edge of $G'$ that $H$ is attached to. 
    Thus, $\operatorname{lir}(^2G)\leq 2$ in such a case.

    Suppose that there is a pendant path $P$ attached to a vertex $v_1$ of degree three (since $G$ is not a path) such that the removal of $P$ and $v_1$ results in a graph without isolated edges (otherwise there is a pendant $K_{1,3}$ or $K_{1,3}$ with some edges subdivided once).
    Denote by $v_0$ and $v_2$ the neighbors of $v_1$ that are not on $P$.
    Denote by $G'$ the graph obtained from $G$ after removing $P$ and $v_1$.
    Since each component of $G'$ is different from $K_2$, $^2G'$ admits a suitable coloring.
    Note that, since vertices of degree three are independent in $G$, both $v_0$ and $v_2$ are of degree one in $G'$.
    Now, we consider every possibility how the multiedges incident to $v_0$ and $v_2$ in $^2G'$ are colored.
    Denote by $e_0$ and $e_2$ the multiedges incident to $v_0$ and $v_2$ in $^2G'$, respectively.
    We distinguish four cases depending on the colors of $e_0$ and $e_2$.
    For an overview of these cases, see Figure~\ref{fig:pendant_path}.

    \begin{figure}[h]
        \centering
        \begin{tikzpicture}
	\begin{pgfonlayer}{nodelayer}
		\node [style={black_dot}] (0) at (-2, 0) {};
		\node [style={black_dot}, label={above:$v_0$}] (1) at (-1, 0) {};
		\node [style={black_dot}, label={above:$v_1$}] (2) at (0, 0) {};
		\node [style={black_dot}, label={above:$v_2$}] (3) at (1, 0) {};
		\node [style={black_dot}] (4) at (2, 0) {};
		\node [style={black_dot}] (5) at (0, -1) {};
		\node [style=empty, label={above:$e_0$}] (6) at (-1.5, 0) {};
		\node [style=empty, label={above:$e_2$}] (7) at (1.5, 0) {};
	\end{pgfonlayer}
	\begin{pgfonlayer}{edgelayer}
		\draw [style={blue_edge}, bend right=15] (0) to (1);
		\draw [style={blue_edge}, bend right=15] (3) to (4);
		\draw [style={blue_edge}, bend left=15] (0) to (1);
		\draw [style={blue_edge}, bend left=15] (3) to (4);
		\draw [style={red_edge}, bend right=15] (1) to (2);
		\draw [style={red_edge}, bend right=15] (2) to (3);
		\draw [style={red_edge}, bend right=15] (2) to (5);
		\draw [style={red_edge}, bend left=15] (1) to (2);
		\draw [style={red_edge}, bend left=15] (2) to (3);
		\draw [style={red_edge}, bend left=15] (2) to (5);
	\end{pgfonlayer}
\end{tikzpicture}%
\hspace{1cm}%
\begin{tikzpicture}
	\begin{pgfonlayer}{nodelayer}
		\node [style={black_dot}] (0) at (-2, 0) {};
		\node [style={black_dot}, label={above:$v_0$}] (1) at (-1, 0) {};
		\node [style={black_dot}, label={above:$v_1$}] (2) at (0, 0) {};
		\node [style={black_dot}, label={above:$v_2$}] (3) at (1, 0) {};
		\node [style={black_dot}, label={below:$\color{red}{\neq3}$}] (4) at (2, 0) {};
		\node [style={black_dot}] (5) at (0, -1) {};
		\node [style=empty, label={above:$e_0$}] (6) at (-1.5, 0) {};
		\node [style=empty, label={above:$e_2$}] (7) at (1.5, 0) {};
	\end{pgfonlayer}
	\begin{pgfonlayer}{edgelayer}
		\draw [style={blue_edge}, bend right=15] (0) to (1);
		\draw [style={red_edge}, bend right=15] (3) to (4);
		\draw [style={blue_edge}, bend left=15] (0) to (1);
		\draw [style={red_edge}, bend left=15] (3) to (4);
		\draw [style={red_edge}, bend right=15] (1) to (2);
		\draw [style={blue_edge}, bend right=15] (2) to (3);
		\draw [style={red_edge}, bend right=15] (2) to (5);
		\draw [style={red_edge}, bend left=15] (1) to (2);
		\draw [style={red_edge}, bend left=15] (2) to (3);
		\draw [style={blue_edge}, bend left=15] (2) to (5);
	\end{pgfonlayer}
\end{tikzpicture}%
\hspace{1cm}%
\begin{tikzpicture}
	\begin{pgfonlayer}{nodelayer}
		\node [style={black_dot}, label={below:$\color{blue}{3}$}] (0) at (-2, 0) {};
		\node [style={black_dot}, label={above:$v_0$}] (1) at (-1, 0) {};
		\node [style={black_dot}, label={above:$v_1$}] (2) at (0, 0) {};
		\node [style={black_dot}, label={above:$v_2$}] (3) at (1, 0) {};
		\node [style={black_dot}, label={below:$\color{red}{3}$}] (4) at (2, 0) {};
		\node [style={black_dot}] (5) at (0, -1) {};
		\node [style=empty, label={above:$e_0$}] (6) at (-1.5, 0) {};
		\node [style=empty, label={above:$e_2$}] (7) at (1.5, 0) {};
	\end{pgfonlayer}
	\begin{pgfonlayer}{edgelayer}
		\draw [style={blue_edge}, bend right=15] (0) to (1);
		\draw [style={red_edge}, bend right=15] (3) to (4);
		\draw [style={blue_edge}, bend left=15] (0) to (1);
		\draw [style={red_edge}, bend left=15] (3) to (4);
		\draw [style={blue_edge}, bend right=15] (1) to (2);
		\draw [style={blue_edge}, bend right=15] (2) to (3);
		\draw [style={blue_edge}, bend right=15] (2) to (5);
		\draw [style={blue_edge}, bend left=15] (1) to (2);
		\draw [style={blue_edge}, bend left=15] (2) to (3);
		\draw [style={blue_edge}, bend left=15] (2) to (5);
	\end{pgfonlayer}
\end{tikzpicture}

\begin{tikzpicture}
	\begin{pgfonlayer}{nodelayer}
		\node [style={black_dot}, label={below:$\color{red}{\neq 3}$}] (0) at (-2, 0) {};
		\node [style={black_dot}, label={above:$v_0$}] (1) at (-1, 0) {};
		\node [style={black_dot}, label={above:$v_1$}] (2) at (0, 0) {};
		\node [style={black_dot}, label={above:$v_2$}] (3) at (1, 0) {};
		\node [style={black_dot}] (4) at (2, 0) {};
		\node [style={black_dot}] (5) at (0, -1) {};
		\node [style=empty, label={above:$e_0$}] (6) at (-1.5, 0) {};
		\node [style=empty, label={above:$e_2$}] (7) at (1.5, 0) {};
	\end{pgfonlayer}
	\begin{pgfonlayer}{edgelayer}
		\draw [style={blue_edge}, bend right=15] (0) to (1);
		\draw [style={blue_edge}, bend right=15] (3) to (4);
		\draw [style={red_edge}, bend left=15] (0) to (1);
		\draw [style={red_edge}, bend left=15] (3) to (4);
		\draw [style={red_edge}, bend right=15] (1) to (2);
		\draw [style={red_edge}, bend right=15] (2) to (3);
		\draw [style={red_edge}, bend right=15] (2) to (5);
		\draw [style={red_edge}, bend left=15] (1) to (2);
		\draw [style={red_edge}, bend left=15] (2) to (3);
		\draw [style={red_edge}, bend left=15] (2) to (5);
	\end{pgfonlayer}
\end{tikzpicture}%
\hspace{1cm}%
\begin{tikzpicture}
	\begin{pgfonlayer}{nodelayer}
		\node [style={black_dot}, label={below:$\color{red}{\neq 3}$}] (0) at (-2, 0) {};
		\node [style={black_dot}, label={above:$v_0$}] (1) at (-1, 0) {};
		\node [style={black_dot}, label={above:$v_1$}] (2) at (0, 0) {};
		\node [style={black_dot}, label={above:$v_2$}] (3) at (1, 0) {};
		\node [style={black_dot}, label={below:$\color{red}{\neq 3}$}] (4) at (2, 0) {};
		\node [style={black_dot}] (5) at (0, -1) {};
		\node [style=empty, label={above:$e_0$}] (6) at (-1.5, 0) {};
		\node [style=empty, label={above:$e_2$}] (7) at (1.5, 0) {};
	\end{pgfonlayer}
	\begin{pgfonlayer}{edgelayer}
		\draw [style={blue_edge}, bend right=15] (0) to (1);
		\draw [style={blue_edge}, bend right=15] (3) to (4);
		\draw [style={red_edge}, bend left=15] (0) to (1);
		\draw [style={red_edge}, bend left=15] (3) to (4);
		\draw [style={red_edge}, bend right=15] (1) to (2);
		\draw [style={red_edge}, bend right=15] (2) to (3);
		\draw [style={red_edge}, bend right=15] (2) to (5);
		\draw [style={red_edge}, bend left=15] (1) to (2);
		\draw [style={red_edge}, bend left=15] (2) to (3);
		\draw [style={red_edge}, bend left=15] (2) to (5);
	\end{pgfonlayer}
\end{tikzpicture}
        \caption{Cases considered in the proof of Theorem~\ref{subcubic_independent}.}
        \label{fig:pendant_path}
    \end{figure}
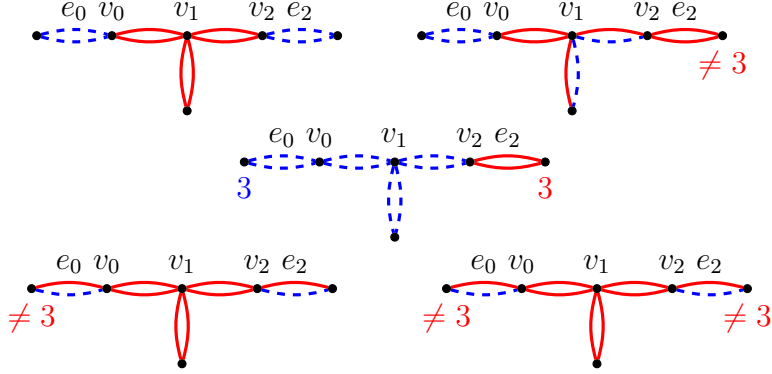

    \textbf{Case 1.}
    Suppose that $e_0$ and $e_2$ are monochromatic, colored with the same color, without loss of generality blue-blue.
    In this case, color multiedges incident to $v_1$ red-red. 
    If $P$ is of odd length, color the remaining multiedges by alternating blue-blue and red-red on pairs of adjacent multiedges, starting with blue-blue.
    If $P$ is of even length, use red-red on the multiedge closest to $v_1$ and then continue as in the previous case.
    Hence, in all these cases, each multiedge of $^2P$ is monochromatic.
    Clearly, the obtained coloring is locally irregular and the pendant multiedge of $^2P$ is not red-blue.

    \textbf{Case 2.}
    Suppose that $e_0$ and $e_2$ are monochromatic, but of different colors. 
    Without loss of generality, let $e_0$ be blue-blue and let $e_2$ be red-red.
    We distinguish three subcases.

     \textbf{Subcase 2.1.}
    If the red degree of the neighbor of $v_2$ is different from three, then color $v_0v_1$ red-red, $v_1v_2$ red-blue, and the edge of $P$ incident to $v_1$ also red-blue.
    If the length of $P$ is odd, color the remaining multiedges by alternating blue-blue and red-red on pairs of adjacent multiedges, starting with blue-blue.
    If the length of $P$ is even, color the uncolored multiedge closest to $v_1$ blue-blue, and then continue as in the previous case, starting with red-red. 
    Hence, in all cases, every multiedge of $^2P$ is monochromatic with the exception of the multiedge incident to $v_1$. 
    However, since the red and blue degrees of $v_1$ are four and two, respectively, the obtained coloring preserves the property that there is not a red-blue pendant multiedge with the end vertex of the red and blue degrees three and three.

    \textbf{Subcase 2.2.}
    If the red degree of the neighbor of $v_2$ is three and the blue degree of the neighbor of $v_0$ is not three, reverse the roles of the blue and red color in Subcase 2.1.

    \textbf{Subcase 2.3.}
    If the blue degree of the neighbor of $v_0$ is not three and the red degree of the neighbor of $v_2$ is not three, then color every multiedge incident to $v_1$ blue-blue.
    Color the remaining multiedges of $^2P$ in a similar way as in Case 1 (reverse the colors on $^2P$).

    \textbf{Case 3.}
    Suppose that $e_0$ is red-blue, and $e_1$ is monochromatic.
    Without loss of generality, let $e_2$ be blue-blue.
    Observe that the red degree of the neighbor of $v_0$ is not three; otherwise the blue degree of such a vertex would be three, but this is not possible due to the additional assumption that the end vertex of a multicolored pendant multiedge (in this case $e_0$ in $^2G'$) does not have the red and blue degrees equal to three and three.
    Hence, in this case, color every multiedge incident to $v_1$ red-red, and the remaining multiedges of $^2P$ as in Case 1.

    \textbf{Case 4.}
    Suppose that both $e_0$ and $e_2$ are red-blue.
    As in Case 3, the neighbors of $v_0$ and $v_2$ do not have red and blue degrees equal to three and three.
    Hence, in this case, color every multiedge incident to $v_1$ red-red, and color the remaining edges of $P$ as in Case 1.

    In the following, we assume that $G$ does not contain a pendant path.
    Suppose that there is a pendant triangle in $G$, i.e., two vertices $v_1, v_2$ of degree two and a vertex $v_3$ of degree three such that $v_1v_2, v_1v_3, v_2v_3$ are edges of $G$.
    Let $G'$ be a graph obtained from $G$ after removing the vertices $v_1,v_2,v_3$.
    Consider the locally irregular 2-coloring of $^2G'$ with red and blue where no red-blue pendant multiedge has an end vertex with the red and blue degrees equal to three and three.
    Denote by $v_4$ the neighbor of $v_3$ different from $v_1$ and $v_2$.
    Note that all neighbors of $v_3$ are of degree two in $G$.
    Hence, $v_4$ is a pendant vertex in $G'$.
    
    If the multiedge incident to $v_4$ in $^2G'$ is monochromatic, without loss of generality blue-blue, color the multiedges $v_3v_4,v_1v_3$ red-red, and the remaining multiedges blue-blue.
    
    If the multiedge incident to $v_4$ is red-blue, the red or the blue degree of the neighbor of $v_4$ is different from three. 
    Without loss of generality suppose that the red degree of the neighbor of $v_4$ is not three.
    Then, once again, color the edges $v_3v_4,v_1v_3$ red-red, and the remaining edges blue-blue.

    Thus, in all cases if $G$ has a pendant triangle, $^2G$ admits a locally irregular edge 2-coloring.

    Therefore, in the following, we suppose that $G$ does not have a pendant triangle.
    Suppose that there are two adjacent vertices $v_1,v_2$ of degree two in $G$.
    Since $G$ is not a cycle, we may suppose that $v_1$ is adjacent to a vertex $v_0$ of degree three.
    Let $G'$ be a graph obtained from $G$ after deleting $v_0$ and $v_1$.
    Consider the locally irregular edge coloring of $^2G'$ with red and blue where no red-blue pendant multiedge has an end vertex with the red and blue degrees equal to three and three.
    If $G'$ is disconnected, from the fact that $G$ does not have a pendant path, it follows that $G'$ does not contain isolated edges. 
    Hence, each component of $^2G'$ admits a suitable locally irregular edge coloring.
    
    Denote by $x_0$ and $y_0$ the neighbors of $v_0$ different from $v_1$.
    Note that $x_0$ and $y_0$ are vertices of degree two in $G$ ($G$ does not contain a pendant vertex).
    Denote by $x_1$ and $y_1$ the neighbors of $x_0$ and $y_0$ in $G'$, respectively (since $G$ does not contain a pendant triangle, both $x_0$ and $y_0$ are different from $v_2$). 
    Denote by $z$ the neighbor of $v_2$ different from $v_1$.
    Similarly to the previous cases discussed in this proof, we distinguish cases depending on colors used on multiedges $x_0x_1$, $y_0y_1$, and $v_2z$.
    For each coloring of each of the considered multiedges, we also distinguish cases depending on the color degrees of vertices $x_1$, $y_1$ and $z$. 
    In general, for each of these three multiedges we have 15 possibilities on how they are colored, and what are the color degrees on one of their ends, see Figure~\ref{fig:15_colorings}.
    Due to the number of all cases considered, we used a computer program that, for each combination of colorings of $x_0x_1, y_0y_1, v_2z$,  goes through every possibility of the coloring of the removed middle part (multiedges incident to $v_1$ and $v_2$) and checks for possible conflicts on the vertices of this subgraph.
    A suitable coloring of the removed multiedges can be found in every case. 

   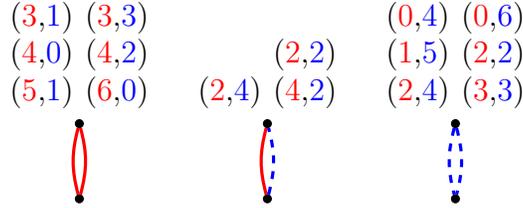
\begin{figure}[h]
        \centering
        \begin{tikzpicture}
	\begin{pgfonlayer}{nodelayer}
		\node [style={black_dot}] (0) at (-2, 1) {};
		\node [style=empty] (1) at (-2, 1) {};
		\node [style=empty, label={above:$\color{black}{(}\color{red}{5}\color{black}{,}\color{blue}{1}\color{black}{)}$}] (6) at (-2.5, 1) {};
		\node [style=empty, label={above:$\color{black}{(}\color{red}{4}\color{black}{,}\color{blue}{2}\color{black}{)}$}] (7) at (-1.5, 1.5) {};
		\node [style=empty, label={above:$\color{black}{(}\color{red}{3}\color{black}{,}\color{blue}{3}\color{black}{)}$}] (9) at (-1.5, 2) {};
		\node [style=empty, label={above:$\color{black}{(}\color{red}{4}\color{black}{,}\color{blue}{0}\color{black}{)}$}] (17) at (-2.5, 1.5) {};
		\node [style=empty, label={above:$\color{black}{(}\color{red}{3}\color{black}{,}\color{blue}{1}\color{black}{)}$}] (18) at (-2.5, 2) {};
		\node [style={black_dot}] (1) at (-2, 0) {};
		\node [style={black_dot}] (2) at (0.5, 1) {};
		\node [style=empty] (10) at (0.5, 1) {};
		\node [style=empty, label={above:$\color{black}{(}\color{red}{4}\color{black}{,}\color{blue}{2}\color{black}{)}$}] (11) at (1, 1) {};
		\node [style=empty, label={above:$\color{black}{(}\color{red}{2}\color{black}{,}\color{blue}{2}\color{black}{)}$}] (16) at (1, 1.5) {};
		\node [style={black_dot}] (3) at (0.5, 0) {};
		\node [style={black_dot}] (4) at (3, 1) {};
		\node [style=empty] (12) at (3, 1) {};
		\node [style=empty, label={above:$\color{black}{(}\color{red}{2}\color{black}{,}\color{blue}{4}\color{black}{)}$}] (13) at (2.5, 1) {};
		\node [style=empty, label={above:$\color{black}{(}\color{red}{1}\color{black}{,}\color{blue}{5}\color{black}{)}$}] (14) at (2.5, 1.5) {};
		\node [style=empty, label={above:$\color{black}{(}\color{red}{0}\color{black}{,}\color{blue}{6}\color{black}{)}$}] (15) at (3.5, 2) {};
		\node [style=empty, label={above:$\color{black}{(}\color{red}{2}\color{black}{,}\color{blue}{2}\color{black}{)}$}] (19) at (3.5, 1.5) {};
		\node [style=empty, label={above:$\color{black}{(}\color{red}{0}\color{black}{,}\color{blue}{4}\color{black}{)}$}] (20) at (2.5, 2) {};
		\node [style={black_dot}] (5) at (3, 0) {};
		\node [style=empty, label={above:$\color{black}{(}\color{red}{6}\color{black}{,}\color{blue}{0}\color{black}{)}$}] (21) at (-1.5, 1) {};
		\node [style=empty, label={above:$\color{black}{(}\color{red}{2}\color{black}{,}\color{blue}{4}\color{black}{)}$}] (22) at (0, 1) {};
		\node [style=empty, label={above:$\color{black}{(}\color{red}{3}\color{black}{,}\color{blue}{3}\color{black}{)}$}] (23) at (3.5, 1) {};
	\end{pgfonlayer}
	\begin{pgfonlayer}{edgelayer}
		\draw [style={red_edge}, bend right=15] (0) to (1);
		\draw [style={red_edge}, bend right=15] (2) to (3);
		\draw [style={blue_edge}, bend right=15] (4) to (5);
		\draw [style={blue_edge}, bend right=15] (3) to (2);
		\draw [style={blue_edge}, bend left=15] (4) to (5);
		\draw [style={red_edge}, bend right=15] (1) to (0);
	\end{pgfonlayer}
\end{tikzpicture}
        \caption{Possibilities of colorings of multiedges $x_0x_1$, $y_0y_1$ and $v_2z$.}
        \label{fig:15_colorings}
    \end{figure}

    Thus, $G$ does not contain adjacent vertices of degree two. Hence, $G$ is a graph obtained from a cubic graph by subdividing each its edge exactly once.
    This, however, means that $G$ is bipartite without pendant edges, and from Theorem \ref{cycle} multigraph $^2G$ has locally irregular coloring using two colors with the desired property.
\end{proof}

\begin{tw}\label{subcubic_long_paths}
Let $G$ be a graph obtained from a cubic graph by replacing some of its edges with paths of length at least five. Then $\operatorname{lir}(^2G)\leq 2$. 
\end{tw}
\begin{proof}
Let $G_1$ be a simple subcubic graph with $\operatorname{lir}(^2G_1) \leq 2$.
Let $u$ and $v$ be the adjacent vertices of degree three in $G_1$.
Denote by $G_2$ the graph obtained from $G_1$ by replacing the edge $uv$ with a path $P$ of length $\ell \geq 5$; let the vertices of this path be $u,x_1,x_2 \dots, x_{\ell - 1},v$.

In this proof, the specific coloring, which we call standard, of a \linebreak 2-multipath is used multiple times. We say that the coloring of a 2-multipath $v_1v_2\dots v_{4k+3}$ of length $4k+2$ is \textit{standard} if all multiedges $v_1v_2,v_2v_3, v_5v_6,v_6v_7, \dots v_{4k+1}v_{4k+2},v_{4k+2}v_{4k+3}$ are colored red-red (or blue-blue), and all its remaining multiedges are colored blue-blue (or red-red).

Consider the locally irregular 2-coloring of $^2G_1$ and two main cases \linebreak according to the coloring of the multiedge $uv$ in $^2G_1$.

\textbf{Case 1.} The multiedge $uv$ is monochromatic in the coloring of $^2G_1$. 
Without loss of generality suppose that $uv$ is blue-blue. 
At least one of the vertices $u$ and $v$ does not have a blue degree four; assume that the blue degree of $u$ is not four.
Color the multiedges $ux_1,x_1x_2,x_{\ell-1},v$ blue-blue. 
The partial coloring is locally irregular. 
The uncolored part of $^2G_2$ is a 2-multipath $^2P'$ of length $\ell-3$.
If $\ell-3 \equiv 2 \pmod{4}$, apply the standard coloring on $^2P'$ starting and ending with red-red multiedges.
If $\ell - 3 \equiv 3 \pmod{4}$, color $x_2x_3$ red-blue, and the rest of $^2P'$ in a standard way starting and ending with red-red multiedges.
If $\ell-3 \equiv 0 \pmod{4}$, color $x_2x_3$ and $x_3x_4$ red-blue, and apply the standard coloring on the rest of $^2P'$ starting and ending with red-red multiedges.
If $\ell-3 \equiv 1 \pmod{4}$, color $x_2x_3$ and $x_3x_4$ red-blue, color $x_4x_5$ blue-blue, and the rest of $^2P'$ in a standard way starting and ending with red-red multiedges.
It is not hard to see that the resulting coloring is locally irregular in all these cases.

\textbf{Case 2.} The multiedge $uv$ is red-blue in the coloring of $^2G_1$. 
Denote by $r_u$ and $r_v$ the red degrees of $u$ and $v$ in the coloring of $^2G_1$.
Since the coloring is locally irregular and a red-blue multiedge joins $u$ and $v$ in $^2G_1$, we get that $r_u \neq r_v$.
Thus, in the following, we may assume that $1 \leq r_u < r_v$.
Color the multiedges $ux_1$ and $x_{\ell-1}v$ red-blue.
Thus, $u$ and $v$ have the same color degrees in the coloring of $^2G_2$ as they have in the coloring of $^2G_1$, i.e., no conflict between $u$ or $v$ and their neighbors outside of $P$ is created.
We distinguish subcases depending on the pair $(r_u,r_v)$ and $\ell$. 



\textbf{Subcase 2.1.} Suppose that $(r_u,r_v) \in \{(1,2),(1,4),(2,4),(2,5),(4,5)\}$. 
If $\ell \leq 8$, we use one of the colorings depicted in Figure~\ref{fig:subcase21}. If $\ell \geq 9$, we extend such a colored 2-multipath by replacing two incident red-red multiedges with a standardly colored 2-multipath of any length $4k+2$ starting and ending with red-red multiedges.
Since $r_u \neq 3$, no conflict between $u$ and its neighbor on $P$ was created. Similarly, $r_v \neq 3$, hence the blue degree of $v$ is different than three, and $v$ is not in conflict with its neighbor on $P$.
Thus, if the resulting coloring is not locally irregular, there is a conflict between two adjacent internal vertices of $P$. 
But this is not the case, which can be easily seen in Figure~\ref{fig:subcase21}.

\begin{figure}[h]
    \centering
    \begin{subfigure}{0.45\textwidth}
        \begin{center}
            \begin{tikzpicture}[scale=0.7]
	\begin{pgfonlayer}{nodelayer}
		\node [style={black_dot}, label={above:$u$}] (0) at (-1, 0) {};
		\node [style={black_dot}] (1) at (0, 0) {};
		\node [style={black_dot}] (2) at (1, 0) {};
		\node [style={black_dot}] (3) at (2, 0) {};
		\node [style={black_dot}] (4) at (3, 0) {};
		\node [style={black_dot}, label={above:$v$}] (5) at (4, 0) {};
	\end{pgfonlayer}
	\begin{pgfonlayer}{edgelayer}
		\draw [style={red_edge}, bend right=15] (0) to (1);
		\draw [style={red_edge}, bend right=15] (4) to (5);
		\draw [style={red_dashed}, bend right=15] (1) to (2);
		\draw [style={red_dashed}, bend right=15] (2) to (3);
		\draw [style={red_dashed}, bend left=15] (1) to (2);
		\draw [style={red_dashed}, bend left=15] (2) to (3);
		\draw [style={blue_edge}, bend right=15] (3) to (4);
		\draw [style={blue_edge}, bend left=15] (3) to (4);
		\draw [style={blue_edge}, bend left=15] (4) to (5);
		\draw [style={blue_edge}, bend left=15] (0) to (1);
	\end{pgfonlayer}
\end{tikzpicture}
        \end{center}
        \caption{$\ell = 5$}
    \end{subfigure}
    \begin{subfigure}{0.45\textwidth}
        \begin{center}
            \begin{tikzpicture}[scale=0.7]
	\begin{pgfonlayer}{nodelayer}
		\node [style={black_dot}, label={above:$u$}] (0) at (-0.5, 0) {};
		\node [style={black_dot}] (1) at (0.5, 0) {};
		\node [style={black_dot}] (2) at (1.5, 0) {};
		\node [style={black_dot}] (3) at (2.5, 0) {};
		\node [style={black_dot}] (4) at (3.5, 0) {};
		\node [style={black_dot}, label={above:$v$}] (5) at (5.5, 0) {};
		\node [style={black_dot}] (6) at (4.5, 0) {};
	\end{pgfonlayer}
	\begin{pgfonlayer}{edgelayer}
		\draw [style={red_edge}, bend right=15] (0) to (1);
		\draw [style={red_dashed}, bend right=15] (1) to (2);
		\draw [style={red_dashed}, bend right=15] (2) to (3);
		\draw [style={red_dashed}, bend left=15] (1) to (2);
		\draw [style={red_dashed}, bend left=15] (2) to (3);
		\draw [style={blue_edge}, bend right=15] (3) to (4);
		\draw [style={blue_edge}, bend left=15] (3) to (4);
		\draw [style={blue_edge}, bend left=15] (0) to (1);
		\draw [style={red_edge}, bend right=15] (6) to (5);
		\draw [style={blue_edge}, bend left=15] (6) to (5);
		\draw [style={blue_edge}, bend left=15] (6) to (4);
		\draw [style={blue_edge}, bend left=15] (4) to (6);
	\end{pgfonlayer}
\end{tikzpicture}
        \end{center}
        \caption{$\ell = 6$}
    \end{subfigure}
    \begin{subfigure}{0.45\textwidth}
        \begin{center}
            \begin{tikzpicture}[scale=0.7]
	\begin{pgfonlayer}{nodelayer}
		\node [style={black_dot}, label={above:$u$}] (0) at (0.5, 0) {};
		\node [style={black_dot}] (1) at (1.5, 0) {};
		\node [style={black_dot}] (2) at (2.5, 0) {};
		\node [style={black_dot}] (3) at (3.5, 0) {};
		\node [style={black_dot}] (4) at (5.5, 0) {};
		\node [style={black_dot}, label={above:$v$}] (5) at (7.5, 0) {};
		\node [style={black_dot}] (6) at (6.5, 0) {};
		\node [style={black_dot}] (7) at (4.5, 0) {};
	\end{pgfonlayer}
	\begin{pgfonlayer}{edgelayer}
		\draw [style={red_edge}, bend right=15] (0) to (1);
		\draw [style={red_dashed}, bend right=15] (1) to (2);
		\draw [style={red_dashed}, bend right=15] (2) to (3);
		\draw [style={red_dashed}, bend left=15] (1) to (2);
		\draw [style={red_dashed}, bend left=15] (2) to (3);
		\draw [style={blue_edge}, bend left=15] (0) to (1);
		\draw [style={red_edge}, bend right=15] (6) to (5);
		\draw [style={blue_edge}, bend left=15] (6) to (5);
		\draw [style={blue_edge}, bend left=15] (6) to (4);
		\draw [style={blue_edge}, bend left=15] (4) to (6);
		\draw [style={red_edge}, bend right=15] (3) to (7);
		\draw [style={blue_edge}, bend left=15] (3) to (7);
		\draw [style={blue_edge}, bend right=15] (7) to (4);
		\draw [style={blue_edge}, bend left=15] (7) to (4);
	\end{pgfonlayer}
\end{tikzpicture}
        \end{center}
        \caption{$\ell = 7$}
    \end{subfigure}
    \begin{subfigure}{0.45\textwidth}
        \begin{center}
            \begin{tikzpicture}[scale=0.7]
	\begin{pgfonlayer}{nodelayer}
		\node [style={black_dot}, label={above:$u$}] (0) at (0.5, 0) {};
		\node [style={black_dot}] (1) at (1.5, 0) {};
		\node [style={black_dot}] (2) at (2.5, 0) {};
		\node [style={black_dot}] (3) at (3.5, 0) {};
		\node [style={black_dot}] (4) at (6.5, 0) {};
		\node [style={black_dot}, label={above:$v$}] (5) at (8.5, 0) {};
		\node [style={black_dot}] (6) at (7.5, 0) {};
		\node [style={black_dot}] (7) at (4.5, 0) {};
		\node [style={black_dot}] (8) at (5.5, 0) {};
	\end{pgfonlayer}
	\begin{pgfonlayer}{edgelayer}
		\draw [style={red_edge}, bend right=15] (0) to (1);
		\draw [style={red_dashed}, bend right=15] (1) to (2);
		\draw [style={red_dashed}, bend right=15] (2) to (3);
		\draw [style={red_dashed}, bend left=15] (1) to (2);
		\draw [style={red_dashed}, bend left=15] (2) to (3);
		\draw [style={blue_edge}, bend left=15] (0) to (1);
		\draw [style={red_edge}, bend right=15] (6) to (5);
		\draw [style={blue_edge}, bend left=15] (6) to (5);
		\draw [style={blue_edge}, bend left=15] (6) to (4);
		\draw [style={blue_edge}, bend left=15] (4) to (6);
		\draw [style={red_edge}, bend right=15] (3) to (7);
		\draw [style={blue_edge}, bend left=15] (3) to (7);
		\draw [style={blue_edge}, bend right=15] (8) to (4);
		\draw [style={blue_edge}, bend left=15] (8) to (4);
		\draw [style={blue_edge}, bend left=15] (7) to (8);
		\draw [style={red_edge}, bend right=15] (7) to (8);
	\end{pgfonlayer}
\end{tikzpicture}
        \end{center}
        \caption{$\ell = 8$}
    \end{subfigure}
    \caption{Coloring of the remaining multiedges in Subcase 2.1 from the proof of Theorem~\ref{subcubic_long_paths}. If $\ell \geq 9$, two adjacent red-red multiedges (dotted) are replaced by a standardly colored 2-multipath of length $4k+2$ starting and ending with red-red multiedges.}
    \label{fig:subcase21}
\end{figure}
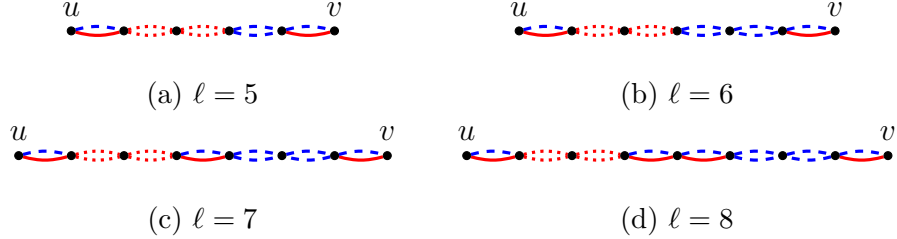

\textbf{Subcase 2.2.} Suppose that $(r_u,r_v) \in \{(3,4),(3,5)\}$.
Similarly to Subcase 2.1, we start with coloring the remaining multiedges in the case when $\ell \leq 8$ (see Figure~\ref{fig:subcase22}, and then we extend this colored 2-multipath if $\ell \geq 9$. 
In this case, we replace two adjacent blue-blue multiedges with a standardly colored 2-multipath of length $4k+2$; hence, we start and end with two blue-blue multiedges.
There is no conflict between $u$ or $v$ and their respective neighbors on $P$, as $r_u = 3$ and the blue degree of $v$ is at most two (since $r_v \geq 4$). 
Since there is no conflict between two adjacent vertices on $P$ (see Figure~\ref{fig:subcase22}), the resulting coloring is locally irregular.
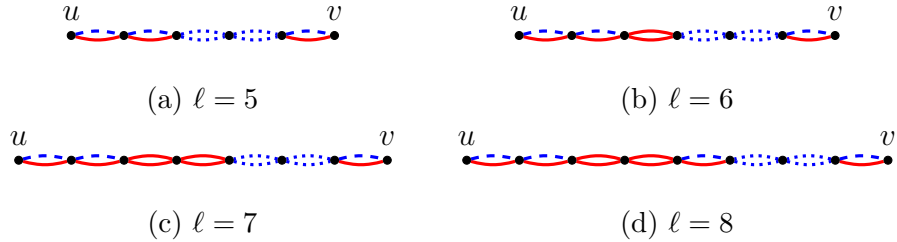
\begin{figure}[h]
    \centering
    \begin{subfigure}{0.45\textwidth}
        \begin{center}
            \begin{tikzpicture}[scale=0.7]
	\begin{pgfonlayer}{nodelayer}
		\node [style={black_dot}, label={above:$u$}] (0) at (-1, 0) {};
		\node [style={black_dot}] (1) at (0, 0) {};
		\node [style={black_dot}] (2) at (1, 0) {};
		\node [style={black_dot}] (3) at (2, 0) {};
		\node [style={black_dot}] (4) at (3, 0) {};
		\node [style={black_dot}, label={above:$v$}] (5) at (4, 0) {};
	\end{pgfonlayer}
	\begin{pgfonlayer}{edgelayer}
		\draw [style={red_edge}, bend right=15] (0) to (1);
		\draw [style={red_edge}, bend right=15] (4) to (5);
		\draw [style={blue_edge}, bend left=15] (4) to (5);
		\draw [style={blue_edge}, bend left=15] (0) to (1);
		\draw [style={blue_edge}, bend left=15] (1) to (2);
		\draw [style={blue_dashed}, bend left=15] (2) to (3);
		\draw [style={blue_dashed}, bend right=15] (2) to (3);
		\draw [style={blue_dashed}, bend left=15] (3) to (4);
		\draw [style={blue_dashed}, bend right=15] (3) to (4);
		\draw [style={red_edge}, bend left=15] (2) to (1);
	\end{pgfonlayer}
\end{tikzpicture}
        \end{center}
        \caption{$\ell = 5$}
    \end{subfigure}
    \begin{subfigure}{0.45\textwidth}
        \begin{center}
            \begin{tikzpicture}[scale=0.7]
	\begin{pgfonlayer}{nodelayer}
		\node [style={black_dot}, label={above:$u$}] (0) at (-0.5, 0) {};
		\node [style={black_dot}] (1) at (0.5, 0) {};
		\node [style={black_dot}] (2) at (1.5, 0) {};
		\node [style={black_dot}] (3) at (2.5, 0) {};
		\node [style={black_dot}] (4) at (3.5, 0) {};
		\node [style={black_dot}, label={above:$v$}] (5) at (5.5, 0) {};
		\node [style={black_dot}] (6) at (4.5, 0) {};
	\end{pgfonlayer}
	\begin{pgfonlayer}{edgelayer}
		\draw [style={red_edge}, bend right=15] (0) to (1);
		\draw [style={blue_edge}, bend left=15] (0) to (1);
		\draw [style={red_edge}, bend right=15] (6) to (5);
		\draw [style={blue_edge}, bend left=15] (6) to (5);
		\draw [style={blue_edge}, bend left=15] (1) to (2);
		\draw [style={red_edge}, bend right=15] (1) to (2);
		\draw [style={red_edge}, bend right=15] (2) to (3);
		\draw [style={red_edge}, bend left=15] (2) to (3);
		\draw [style={blue_dashed}, bend right=15] (3) to (4);
		\draw [style={blue_dashed}, bend right=15] (4) to (6);
		\draw [style={blue_dashed}, bend left=345] (6) to (4);
		\draw [style={blue_dashed}, bend right=15] (4) to (3);
	\end{pgfonlayer}
\end{tikzpicture}
        \end{center}
        \caption{$\ell = 6$}
    \end{subfigure}
    \begin{subfigure}{0.45\textwidth}
        \begin{center}
            \begin{tikzpicture}[scale=0.7]
	\begin{pgfonlayer}{nodelayer}
		\node [style={black_dot}, label={above:$u$}] (0) at (0.5, 0) {};
		\node [style={black_dot}] (1) at (1.5, 0) {};
		\node [style={black_dot}] (2) at (2.5, 0) {};
		\node [style={black_dot}] (3) at (3.5, 0) {};
		\node [style={black_dot}] (4) at (5.5, 0) {};
		\node [style={black_dot}, label={above:$v$}] (5) at (7.5, 0) {};
		\node [style={black_dot}] (6) at (6.5, 0) {};
		\node [style={black_dot}] (7) at (4.5, 0) {};
	\end{pgfonlayer}
	\begin{pgfonlayer}{edgelayer}
		\draw [style={red_edge}, bend right=15] (0) to (1);
		\draw [style={blue_edge}, bend left=15] (0) to (1);
		\draw [style={red_edge}, bend right=15] (6) to (5);
		\draw [style={blue_edge}, bend left=15] (6) to (5);
		\draw [style={blue_dashed}, bend right=15] (7) to (4);
		\draw [style={blue_dashed}, bend right=15] (4) to (6);
		\draw [style={blue_dashed}, in=15, out=525] (6) to (4);
		\draw [style={blue_dashed}, bend left=345] (4) to (7);
		\draw [style={blue_edge}, bend left=15] (1) to (2);
		\draw [style={red_edge}, bend right=15] (1) to (2);
		\draw [style={red_edge}, bend right=15] (2) to (3);
		\draw [style={red_edge}, bend right=15] (3) to (7);
		\draw [style={red_edge}, bend left=15] (3) to (7);
		\draw [style={red_edge}, bend left=15] (2) to (3);
	\end{pgfonlayer}
\end{tikzpicture}
        \end{center}
        \caption{$\ell = 7$}
    \end{subfigure}
    \begin{subfigure}{0.45\textwidth}
        \begin{center}
            \begin{tikzpicture}[scale=0.7]
	\begin{pgfonlayer}{nodelayer}
		\node [style={black_dot}, label={above:$u$}] (0) at (0.5, 0) {};
		\node [style={black_dot}] (1) at (1.5, 0) {};
		\node [style={black_dot}] (2) at (2.5, 0) {};
		\node [style={black_dot}] (3) at (3.5, 0) {};
		\node [style={black_dot}] (4) at (6.5, 0) {};
		\node [style={black_dot}, label={above:$v$}] (5) at (8.5, 0) {};
		\node [style={black_dot}] (6) at (7.5, 0) {};
		\node [style={black_dot}] (7) at (4.5, 0) {};
		\node [style={black_dot}] (8) at (5.5, 0) {};
	\end{pgfonlayer}
	\begin{pgfonlayer}{edgelayer}
		\draw [style={red_edge}, bend right=15] (0) to (1);
		\draw [style={blue_edge}, bend left=15] (0) to (1);
		\draw [style={red_edge}, bend right=15] (6) to (5);
		\draw [style={blue_edge}, bend left=15] (6) to (5);
		\draw [style={blue_edge}, bend left=15] (7) to (8);
		\draw [style={red_edge}, bend right=15] (7) to (8);
		\draw [style={blue_dashed}, bend right=15] (8) to (4);
		\draw [style={blue_dashed}, bend right=15] (4) to (6);
		\draw [style={blue_dashed}, bend left=345] (6) to (4);
		\draw [style={blue_dashed}, bend right=15] (4) to (8);
		\draw [style={red_edge}, bend right=15] (1) to (2);
		\draw [style={red_edge}, bend left=15] (2) to (3);
		\draw [style={red_edge}, bend left=15] (3) to (7);
		\draw [style={red_edge}, bend left=15] (7) to (3);
		\draw [style={red_edge}, bend left=15] (3) to (2);
		\draw [style={blue_edge}, bend left=15] (1) to (2);
	\end{pgfonlayer}
\end{tikzpicture}
        \end{center}
        \caption{$\ell = 8$}
    \end{subfigure}
    \caption{Coloring of the remaining multiedges in Subcase 2.2 from the proof of Theorem~\ref{subcubic_long_paths}. If $\ell \geq 9$, two adjacent blue-blue multiedges (dotted) are replaced by a standardly colored 2-multipath of length $4k+2$ starting and ending with two adjacent blue-blue multiedges.}
    \label{fig:subcase22}
\end{figure}

\textbf{Subcase 2.3.} Suppose that $(r_u,r_v) \in \{(1,3),(2,3)\}$.
Color the multiedges of $^2P$ as is depicted in Figure~\ref{fig:subcase23}.
There is no conflict between $u$ and its neighbor on $P$, since the red degree of $u$ is at most two, and the blue degree of $u$ is at least four.
Similarly, there is no conflict between $v$ and its neighbor on $P$, since the red and blue degrees of $v$ are three and three.
Moreover, there is no conflict between two adjacent internal vertices of $P$, hence, the coloring is locally irregular.

\begin{figure}[h]
    \centering
    \begin{subfigure}{0.45\textwidth}
        \begin{center}
            \begin{tikzpicture}[scale=0.7]
	\begin{pgfonlayer}{nodelayer}
		\node [style={black_dot}, label={above:$u$}] (0) at (-1, 0) {};
		\node [style={black_dot}] (1) at (0, 0) {};
		\node [style={black_dot}] (2) at (1, 0) {};
		\node [style={black_dot}] (3) at (2, 0) {};
		\node [style={black_dot}] (4) at (3, 0) {};
		\node [style={black_dot}, label={above:$v$}] (5) at (4, 0) {};
	\end{pgfonlayer}
	\begin{pgfonlayer}{edgelayer}
		\draw [style={red_edge}, bend right=15] (0) to (1);
		\draw [style={red_edge}, bend right=15] (4) to (5);
		\draw [style={blue_edge}, bend left=15] (4) to (5);
		\draw [style={blue_edge}, bend left=15] (0) to (1);
		\draw [style={red_dashed}, bend right=15] (1) to (2);
		\draw [style={red_dashed}, bend left=15] (1) to (2);
		\draw [style={red_dashed}, bend right=15] (2) to (3);
		\draw [style={red_dashed}, bend left=15] (2) to (3);
		\draw [style={red_edge}, bend right=15] (3) to (4);
		\draw [style={blue_edge}, bend left=15] (3) to (4);
	\end{pgfonlayer}
\end{tikzpicture}
        \end{center}
        \caption{$\ell = 5$}
    \end{subfigure}
    \begin{subfigure}{0.45\textwidth}
        \begin{center}
            \begin{tikzpicture}[scale=0.7]
	\begin{pgfonlayer}{nodelayer}
		\node [style={black_dot}, label={above:$u$}] (0) at (-0.5, 0) {};
		\node [style={black_dot}] (1) at (0.5, 0) {};
		\node [style={black_dot}] (2) at (1.5, 0) {};
		\node [style={black_dot}] (3) at (2.5, 0) {};
		\node [style={black_dot}] (4) at (3.5, 0) {};
		\node [style={black_dot}, label={above:$v$}] (5) at (5.5, 0) {};
		\node [style={black_dot}] (6) at (4.5, 0) {};
	\end{pgfonlayer}
	\begin{pgfonlayer}{edgelayer}
		\draw [style={red_edge}, bend right=15] (0) to (1);
		\draw [style={blue_edge}, bend left=15] (0) to (1);
		\draw [style={red_edge}, bend right=15] (6) to (5);
		\draw [style={blue_edge}, bend left=15] (6) to (5);
		\draw [style={red_edge}, bend right=15] (4) to (6);
		\draw [style={blue_edge}, bend right=15] (3) to (4);
		\draw [style={blue_edge}, in=165, out=15] (3) to (4);
		\draw [style={blue_edge}, bend left=15] (4) to (6);
		\draw [style={red_dashed}, bend right=15] (1) to (2);
		\draw [style={red_dashed}, bend right=15] (2) to (3);
		\draw [style={red_dashed}, bend left=345] (3) to (2);
		\draw [style={red_dashed}, in=375, out=165] (2) to (1);
	\end{pgfonlayer}
\end{tikzpicture}
        \end{center}
        \caption{$\ell = 6$}
    \end{subfigure}
    \begin{subfigure}{0.45\textwidth}
        \begin{center}
            \begin{tikzpicture}[scale=0.7]
	\begin{pgfonlayer}{nodelayer}
		\node [style={black_dot}, label={above:$u$}] (0) at (0.5, 0) {};
		\node [style={black_dot}] (1) at (1.5, 0) {};
		\node [style={black_dot}] (2) at (2.5, 0) {};
		\node [style={black_dot}] (3) at (3.5, 0) {};
		\node [style={black_dot}] (4) at (5.5, 0) {};
		\node [style={black_dot, label={above:$v$}}] (5) at (7.5, 0) {};
		\node [style={black_dot}] (6) at (6.5, 0) {};
		\node [style={black_dot}] (7) at (4.5, 0) {};
	\end{pgfonlayer}
	\begin{pgfonlayer}{edgelayer}
		\draw [style={red_edge}, bend right=15] (0) to (1);
		\draw [style={blue_edge}, bend left=15] (0) to (1);
		\draw [style={red_edge}, bend right=15] (6) to (5);
		\draw [style={blue_edge}, bend left=15] (6) to (5);
		\draw [style={red_dashed}, bend left=15] (1) to (2);
		\draw [style={red_dashed}, bend left=15] (2) to (3);
		\draw [style={red_dashed}, bend left=15] (3) to (2);
		\draw [style={red_dashed}, bend left=15] (2) to (1);
		\draw [style={blue_edge}, bend left=15] (3) to (7);
		\draw [style={blue_edge}, bend left=15] (7) to (4);
		\draw [style={blue_edge}, bend left=15] (4) to (7);
		\draw [style={blue_edge}, bend right=345] (7) to (3);
		\draw [style={blue_edge}, bend left=15] (4) to (6);
		\draw [style={red_edge}, bend left=15] (6) to (4);
	\end{pgfonlayer}
\end{tikzpicture}
        \end{center}
        \caption{$\ell = 7$}
    \end{subfigure}
    \begin{subfigure}{0.45\textwidth}
        \begin{center}
            \begin{tikzpicture}[scale=0.7]
	\begin{pgfonlayer}{nodelayer}
		\node [style={black_dot}, label={above:$u$}] (0) at (0.5, 0) {};
		\node [style={black_dot}] (1) at (1.5, 0) {};
		\node [style={black_dot}] (2) at (2.5, 0) {};
		\node [style={black_dot}] (3) at (3.5, 0) {};
		\node [style={black_dot}] (4) at (6.5, 0) {};
		\node [style={black_dot}, label={above:$v$}] (5) at (8.5, 0) {};
		\node [style={black_dot}] (6) at (7.5, 0) {};
		\node [style={black_dot}] (7) at (4.5, 0) {};
		\node [style={black_dot}] (8) at (5.5, 0) {};
	\end{pgfonlayer}
	\begin{pgfonlayer}{edgelayer}
		\draw [style={red_edge}, bend right=15] (0) to (1);
		\draw [style={blue_edge}, bend left=15] (0) to (1);
		\draw [style={red_edge}, bend right=15] (6) to (5);
		\draw [style={blue_edge}, bend left=15] (6) to (5);
		\draw [style={red_dashed}, bend right=15] (1) to (2);
		\draw [style={red_dashed}, bend right=15] (2) to (3);
		\draw [style={red_dashed}, bend right=15] (3) to (2);
		\draw [style={red_dashed}, bend right=15] (2) to (1);
		\draw [style={blue_edge}, bend right=15] (3) to (7);
		\draw [style={blue_edge}, bend right=15] (7) to (8);
		\draw [style={blue_edge}, bend right=15] (8) to (7);
		\draw [style={blue_edge}, bend right=15] (7) to (3);
		\draw [style={red_edge}, bend right=15] (8) to (4);
		\draw [style={red_edge}, bend right=15] (4) to (8);
		\draw [style={red_edge}, bend right=15] (4) to (6);
		\draw [style={blue_edge}, bend left=345] (6) to (4);
	\end{pgfonlayer}
\end{tikzpicture}
        \end{center}
        \caption{$\ell = 8$}
    \end{subfigure}
    \caption{Coloring of the remaining multiedges in Subcase 2.3 from the proof of Theorem~\ref{subcubic_long_paths}. If $\ell \geq 9$, two adjacent red-red multiedges (dotted) are replaced by a standardly colored 2-multipath of length $4k+2$ starting and ending with red-red multiedges.}
    \label{fig:subcase23}
\end{figure}
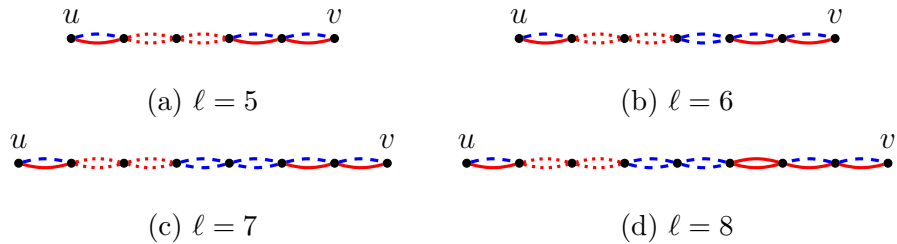

Hence, in all cases, we showed how to extend the locally irregular \linebreak 2-coloring of $^2G_1$ to a locally irregular coloring of $^2G_2$; which completes the proof.
\end{proof}

\section*{Acknowledgments}
We want to sincerely thank Jakub Przybyło for all the effort he put into helping us prove an important result in this paper.

\end{document}